\documentclass{amsart}

\usepackage{amsmath,amscd}
\usepackage{amsthm}
\usepackage{amssymb}
\usepackage{setspace}
\usepackage{fancyhdr}
\usepackage{stmaryrd}
\usepackage{bbm}
\input xy
\xyoption{all}
\usepackage[unicode]{hyperref}
\usepackage{color}

\newtheorem{theo}{{Theorem}}[section]
\newtheorem{coro}[theo]{{Corollary}}
\newtheorem{lemma}[theo]{{Lemma}}
\newtheorem{prop}[theo]{Proposition}

\newtheorem{rk}[theo]{{Remark}}

\theoremstyle{definition}
\newtheorem{defn}[theo]{Definition}

\renewcommand{\phi}{\varphi}

\newcommand{\QQ}{\mathbf{Q}}
\newcommand{\ZZ}{\mathbf{Z}}
\newcommand{\RR}{\mathbf{R}}

\newcommand{\cR}{\mathcal{R}}
\newcommand{\cH}{\mathcal{H}}

\newcommand{\cA}{\mathcal{A}}
\newcommand{\fA}{\mathfrak{A}}
\newcommand{\cZ}{\mathcal{Z}}

\newcommand{\fa}{\mathfrak{a}}

\newcommand{\fX}{\mathfrak{X}}
\newcommand{\fU}{\mathfrak{U}}

\newcommand{\fY}{\mathfrak{Y}}

\newcommand{\cO}{\mathcal{O}}

\newcommand{\cM}{\mathcal{M}}
\newcommand{\fM}{\mathfrak{M}}

\newcommand{\cF}{\mathcal{F}}
\newcommand{\CC}{\mathbf{C}}

\newcommand{\TT}{\mathbf{T}}
\newcommand{\GG}{\mathbf{G}}

\newcommand{\la}{\lambda}

\newcommand{\ra}{\rightarrow}
\newcommand{\lra}{\longrightarrow}

\DeclareMathOperator{\tr}{Tr}
\DeclareMathOperator{\Spec}{Spec}
\DeclareMathOperator{\Ker}{Ker}

\DeclareMathOperator{\Sp}{Sp}
\DeclareMathOperator{\Spf}{Spf}
\DeclareMathOperator{\End}{End}
\DeclareMathOperator{\Hom}{Hom}

\DeclareMathOperator{\GSp}{GSp}

\DeclareMathOperator{\GL}{GL}
\DeclareMathOperator{\Lie}{Lie}
\DeclareMathOperator{\Res}{Res}

\numberwithin{equation}{subsection}

\begin{document}
\author{Chung Pang Mok, Fucheng Tan}
\title{Overconvergent family of Siegel-Hilbert modular forms}

\date{}

\maketitle

\thispagestyle{plain}

 \begin{abstract}
 
We construct one-parameter families of overconvergent Siegel-Hilbert modular forms. This result has applications to construction of Galois representations for automorphic forms of non-cohomological weights, for instance, in the work  \cite{Mok}.
 
 \end{abstract}

\tableofcontents

 \section{Introduction}

The study of $p$-adic families of automorphic forms has been carried out in many works. In the case of elliptic modular forms, the overconvergent modular eigenforms of \emph{finite slope} (i.e. with non-zero Hecke eigenvalue at $p$) are interpolated to be points on a rigid analytic curve, which is known as the Coleman-Mazur eigencurve \cite{CM}.  Before this seminal work, the family of ordinary eigenforms was obtained by Hida \cite{Hi86}.  

Among all the approaches to the construction of eigenvarieties for more general algebraic groups, the work of Kisin-Lai \cite{KL} on overconvergent Hilbert modular forms is most closely related to ours.  Their method is a generalization of that of Coleman-Mazur. In both cases, the key point for interpolating modular forms is the \emph{complete continuity} (cf. P. 425 \cite{Co} for definition) of the Atkin-Lehner operator on certain spaces of overconvergent forms.  
In the case of elliptic modular forms,  Coleman-Mazur interpolate modular forms by twisting by $p$-adic analytic families of Eisenstein series. However, in the more general (Siegel-)Hilbert modular case such a theory of Eisenstein series is not yet available. Instead, one lifts (a certain power of) the Hasse invariant in characteristic $p$ to be a global section of certain automorphic line bundle over the integral model of the Shimura variety. 

We would like to mention certain differences between our method and that of \cite{KL}, which are mainly caused by the generality of the Siegel-Hilbert moduli space.

In the Hilbert modular case of \cite{KL}, they glue the toroidal compactfication of the Rapoport model \cite{Ra} with the Deligne-Pappas model \cite{DP}, because the Rapoport model may not be proper at the places which are ramified in the totally real field. Fortunately, Rapoport's toroidal compactification can be used because the Lie algebra condition,  which causes non-properness at finite distance, is automatic in the boundary. In the Siegel-Hilbert case, one has to do more to take care of the ramified places. There exists the canonical integral model of Pappas-Rapoport \cite{PR2} in the Siegel-Hilbert modular case, which has a moduli interpretation. Its toroidal compactifications are, however, not completely understood.  Fortunately, the (partial) toroidal compactifications and minimal compactification of the ordinary locus is  constructed in \cite{La12} successfully, which will be enough for our use.

Furthermore, we follow the idea of Hida \cite{Hi} to form the formal Igusa tower over the formal completion of the (compactified) moduli space with level structure away from $p$, instead of using the ``unramified $\Gamma_{00}(Np^n)$ cusps" in \cite{KL}.  This seems more convenient in the general Siegel-Hilbert case.

Write $G=\mathrm{Res}_{\cO_F/\ZZ}\GSp_{2g}$. The moduli space above is actually for its subgroup $G'=G\times_{\mathrm{Res}_{F/\QQ}\GG_m}\GG_m$. Finally, with these strategies and results, we construct one-dimensional families of eigenforms on $G'$,  for any totally real field $F$ and $g \geq 1$. More precisely, we obtain, for each classical weight $\kappa$, a reduced rigid analytic curve $\mathcal{E}_{\kappa}$, whose points are in one-to-one correspondence with systems of Hecke eigenvalues of overconvergent automorphic  forms on $G'$, whose weights ``differ" from that of $\kappa$ by parallel weights. One of the key properties of the rigid curve $\mathcal{E}_{\kappa}$ is that the canonical map to the weight space given by  weights of modular forms is, locally in the domain, finite flat. We refer the reader to Theorem \ref{curve} for more details.  

Essentially due to the (local) finite flatness of the weight map on $\mathcal{E}_{\kappa}$ (and the argument in Section \ref{des}),  we have 

\begin{theo}[Theorem \ref{accum}]\label{1}
Let $f$ be a (classical) Siegel-Hilbert modular eigenform on $G=\mathrm{Res}_{\cO_F/\ZZ}\GSp_{2g}$ of weight $\kappa$ with some tame level 
and level $p^n$ at $p$. For any positive integers $t$ with large enough $p$-adic valuation, there exist Siegel-Hilbert modular
eigenforms $f_t$ of the same level and of varying weights, whose Hecke eigenvalues converge $p$-adically to that of
$f$ when $t$ goes to zero $p$-adically.
\end{theo}

This theorem is sufficient for some applications. For example, Theorem \ref{1} is one of the main ingredients for attaching Galois representations to automorphic forms $\pi$ on $\GL_2$ over arbitrary CM fields,  the work of one of us \cite{Mok}. More precisely, in order to construct such a $2$-dimensional representation, one first lifts $\pi$ to an automorphic form $\Pi$ (of non-cohomological type!) on $\GSp_{4}(\mathbf{A}_F)$. Then the Galois representation $\rho_{\Pi}$ associated to $\Pi$ is obtained by interpolating Galois representations associated to forms on $\GSp_{4}(\mathbf{A}_F)$ of cohomological type, with the family of cohomological forms supplied by Theorem \ref{1}. As is mentioned in \cite{Mok}, the use of $p$-adic analytic family of automorphic forms, compared to the use of congruence relations between them, has the advantage that this (less elementary) method allows us to prove local-global compatibility.

We would like to mention that  the eigenvariety for Siegel modular forms (not necessarily with parallel weights) was recently developed in \cite{AIP}, for $F=\QQ$.  Our result, on the other hand, applies to all totally real fields.

 The paper is organized as follows. 
 
 In Section $2$, we recall the results on  integral models of PEL Shimura varieties and their compactificatoins.  In the next section,  we  use the idea of Hida to form the formal Igusa tower. In the last section, we form the spaces of overconvergent Siegel-Hilbert modular eigenforms and then prove that the $U_{(p)}$-operator  is completely continuous on the spaces. Finally by the machinery in \cite{CM}, we construct the rigid curves interpolating these overconvergent forms, and then show Theorem \ref{1}.

\section*{Notation}

\begin{itemize}

\item $F$ is a totally real field of degree $d$ over $\QQ$, and $\cO=\cO_F$ is the ring of integers. We denote by $\mathbf{A}=\mathbf{A}_F$ the ring of adeles of $F$, and by $\mathbf{A}_{f}$ the ring of finite adeles of $F$.   

\item For a maximal torus $T=T_{G}$ of a reductive group $G$ over $\ZZ$, $\mathrm{Nm}: \mathrm{Res}_{\ZZ}^{\cO}T  \ra T$ is the norm map, i.e. for any ring $R$, $\mathrm{Nm}(R): T(\cO\otimes_{\ZZ}R)\ra T(R)$ is given by the norm $N_{F/\QQ}$ on $F$. 

\item $p\geq 2$ is a fixed rational prime.

\item If $K/\QQ_p$ is a finite extension, $K_0$ is the maximal unramified extension of $\QQ_p$ in $K$, and $[K:K_0]=e$. $\bar{\QQ}_p$ is a fixed algebraic closure of $K$, and $\CC_p$ is the completion of $\bar{\QQ}_p$ for the $p$-adic topology.

\item If $H \subset G(\mathbf{A}_f)$ denotes an open compact subgroup, then we write it in the form $H = H_pH^p$, where $H_p \subset G(\QQ_p)$,  $H^p \subset G(\mathbf{A}_f^p)$, for $\mathbf{A}_f^p$ the ring of finite adeles over $F$ with trivial $p$-component.

\end{itemize}

 \section{Siegel-Hilbert moduli spaces}

 \subsection{PEL datum}

 \subsubsection{The general integral PEL data}
 
 We recall the (integral) PEL datum $(\cO_B,*, L, \psi, h)$, whose rational part $(B,*,L_{\QQ}, \psi_{\QQ},h)$ can give rise to a Shimura datum  by 4.1 \cite{Ko}.
 
 \begin{itemize}
\item  $B$ is a finite dimensional  semisimple $\QQ$-algebra whose center is denoted by $F$, and is equipped with a positive involution $\ast$: \[ (ab)^{*}=b^*a^*, b^{**}=b, \quad \forall a,b\in B,\] \[\tr_{B/\QQ}(bb^*)>0, \quad \forall b\neq 0.\] $\cO_B$ is an order of $B$ stabilized by the involution above. 
 
\item  $(L,\psi)$ is a symplectic $(\cO_B,*)$-module over $\ZZ$, i.e. $L$ is a finite free $\ZZ$-module carrying an alternating form $\psi: L\times L\ra \ZZ$, such that \[\psi(bx,y)=\psi(x,b^*y), \qquad \forall x,y\in L, b\in \cO_B.\]

Let  $G$ be the group over $\ZZ$ so that for any $\ZZ$-algebra $R$, 
 \[G(R)=\{g\in \GL_{(\cO_B)_R}(L_R)|\psi(gx,gy)=\nu(g)\psi(x,y), \nu(g)\in R\}.\]

 \item Let \[\tilde{h}: \CC\ra \End_{(\cO_B)_{\RR}}(L_{\RR})\] be an $\RR$-algebra homomorphism  that gives a Hodge structure of type $(1,0),(0,1)$ on $L_{\RR}$, such that $\psi(x,\tilde{h}(\sqrt{-1})y)$ is a symmetric positive definite bilinear form on $L_{\RR}$. The restriction $\tilde{h}|_{\CC^{\times}}$ can be viewed as a morphism of $\RR$-algebraic groups \[h: \Res_{\CC/\RR}\mathbf{G}_{m,\CC}\lra G_{\RR}.\]

 \end{itemize}

 The action of $h$ gives a decomposition
 \begin{equation}\label{v0}
 L_{\CC}=V_{0,\CC}\oplus V_{1,\CC}
 \end{equation} where $h$ acts on the first factor by the character $z\mapsto \bar{z}$ and on the second one by $z\mapsto z$. The Shimura field is then by definition 
 \[E=F[\tr_{\CC}(b|V_{0,\CC}),b\in B].\] The decomposition (\ref{v0}) is then defined over the subfield $E$ of $\CC$.

  \subsubsection{PEL data for symplectic groups}\label{pelsym}

  Let $B=F$ be a totally real field of degree $d$. Let $\cO_B=\cO$ and $*=\mathrm{Id}$ be the trivial involution.  Let $L$ be a finite free $\ZZ$-module of rank $2dg$ equipped with an $\cO$-module structure, together  with the standard symplectic form \[\phi: L\times L\ra \cO\] given by the  antisymmetric matrix $J=\left(
                       \begin{array}{cc}
                           & -I_{dg} \\
                          I_{dg} &   \\
                       \end{array}
                     \right)$.
Set \[\psi=\mathrm{Tr}_{\cO/\ZZ}\circ\phi.\]             

The $\CC$-algebra homomorphism $\tilde{h}$ is \[a+bi\mapsto \begin{pmatrix}
aI_{dg}&-bI_{dg}\\
bI_{dg}& aI_{dg}\end{pmatrix}.\]

We have the PEL datum $(\cO, \mathrm{Id}, L, \psi,h=\tilde{h}|_{\CC^{\times}})$. In this case \[G=\mathrm{Res}_{\cO/\ZZ}\GSp_{2g},\] where $\GSp_{2g}$ is the
split reductive group of symplectic similitudes respecting the matrix $J$.                     
                     
The Shimura field in this case is $E=\QQ$.    
 
\subsection{The Siegel-Hilbert moduli space over the Shimura field}

Keep the Shimura datum $(\cO, \mathrm{Id}, L, \psi,h)$ as above.  Let $H\subset G(\hat{\ZZ})$ be an open compact subgroup. We recall the moduli problem from Section 5 \cite{Ko} and 1.4.1.4 \cite{La}.

Let $\cM_H$ be the functor that assigns to a $\QQ$-scheme $S$ the isomorphism classes of the tuples $(A, i, \la, \alpha_H)$ of the following kind
 
\begin{itemize}
\item $A$ is an abelian scheme over $S$ of relative dimension $dg$, equipped with an $\cO$-action called the \emph{real multiplication}: $i: \cO\hookrightarrow \mathrm{End}_S(A)$.

\item The requirement of Kottwitz determinant condition \[\mathrm{det}_{\cO_S}(b|\Lie A)=\mathrm{det}_{E}(b|V_0), \quad\forall b\in F \]as \emph{polynomial functions}, for which both sides of the equality are considered as morphisms of $S$-schemes.  (cf. Section 5 \cite{Ko}  for details.)

\item $\la: A\ra A^{\vee}$ is a polarization. Recall that a symmetric homomorphism $A\ra A^{\vee}$ is a \emph{polarization} if (locally for the \'etale topology) it comes from a line bundle over $A$ which is ample over $S$. (cf. 6.2 \cite{GIT}.)

\item $\alpha_H$ is an $H$-level structure of type $(L_{\hat{\ZZ}}, \psi)$ analogous to that  defined in Sec. \ref{leveln}. (cf. 1.3.7.6 \cite{La} for more details.) 
\end{itemize}

The functor $\cM_H$ is represented by a separated smooth algebraic stack of finite type over $E=\QQ$, by Artin's theory and Grothendieck's theory of Hilbert schemes. We denote the moduli stack by $\cM_H$ again. If $H$ is neat, then $\cM_H$ is  a smooth quasi-projective  scheme over $\QQ$, by \cite{GIT} (and the theory of Hilbert schemes). As a special case of the construction of $\cM_H$, we have the functor $\cM_{H^p}$ with the level structure $\alpha_H$ being the prime to $p$ level structure on $H^p$. 

We denote the universal abelian scheme over $\cM_H$ by $\cA$, and denote by $\omega$ the pull-back along the unit section  of the relative differentials $\Omega_{\cA/\cM_H}^1$. 
 
\begin{rk}\label{1118}
Let $X$ denote the $G(\RR)$-conjugacy classes of $\tilde{h}$. The complex manifold $G(\QQ)\backslash X\times G(\mathbf{A}^{\infty})/H$ descents to a quasi-projective scheme $Sh_H$ over $\QQ$, which is commonly called the Shimura variety. We have a canonical open and closed immersion \[Sh_H\hookrightarrow [\cM_H]\] of the Shimura variety into the coarse moduli space of the algebraic stack $\cM_H$.  The moduli $[\cM_H]$ is in fact the Shimura variety for the group $G'=G\times_{\mathrm{Res}_{F/\QQ}\GG_m}\GG_m$, the subgroup of $G$ whose determinants lie in $\GG_m$. 
 \end{rk}

 \subsection{Integral models and compactifications}\label{moduli2}
 
 In \cite{La12}, Lan constructs  a normal and flat algebraic stack $\vec{\cM}_H$ over $\ZZ_{(p)}$ which comes with a canonical isomorphism 
 \[\vec{\cM}_H\times_{\Spec\ZZ_{(p)}}\Spec\QQ\simeq \cM_H.\] We recall the construction briefly.

One first  finds an auxiliary  Shimura datum which can provide the canonical integral model and toroidal compactification. In fact, one can embed the $\ZZ$-module $L$ into another finite free $\ZZ$-module $L_{\rm aux}$ which comes with an alternating pairing $\psi_{\rm aux}$ whose restriction to $L$ is $\psi$. The $\RR$-algebra homomorphism $\tilde{h}$ then induces another $\RR$-algebra homomorphism $\tilde{h}_{\rm aux}$, whose restriction to $\CC^{\times}$ is denoted by $h_{\rm aux}$. Moreover, we have a subring $\cO_{\rm aux}\subset\cO$ for which the embedding $L\hookrightarrow L_{\rm aux}$ is $\cO_{\rm aux}$-linear.  The point is that, for the auxiliary Shimura datum $(\cO_{\rm aux}, \mathrm{Id},L_{\rm aux}, \psi_{\rm aux}, h_{\rm aux})$, the prime $p$ is a \emph{good} prime to which the main results of \cite{La} apply.

Now we have an induced homomorphism of algebraic groups over $\ZZ$ \[t: G\lra G_{\rm aux}\] where the second group is defined by the auxiliary Shimura datum in the same way as before. The auxiliary Shimura datum provides a moduli stack $\cM_{G_{\rm aux}(\hat{\ZZ}^p)}$ which is separated smooth and of finite type over $\ZZ_{(p)}$. By the fact that $p$ is a good prime for $\cM_{G_{\rm aux}(\hat{\ZZ}^p)}$, one can show that there is a canonical isomorphism \[\cM_{G_{\rm aux}(\hat{\ZZ}^p)\times G_{\rm aux}(\ZZ_p)}\simeq \cM_{G_{\rm aux}(\hat{\ZZ}^p)}\otimes_{\ZZ_{(p)}}\QQ.\] 

More generally, for any open compact subgroup $H_{\rm aux}=H_{\rm aux}^pG_{\rm aux}(\ZZ_p)\subset G_{\rm aux}(\hat{\ZZ})$ such that $H^p$ is mapped to $H_{\rm aux}^p$ under the morphism $t: G(\ZZ^p)\ra G_{\rm aux}(\ZZ^p)$, we have similarly a moduli stack $\cM_{H_{\rm aux}^p}$ for which $p$ is  a good prime and a morphism 
\begin{equation}\label{418}
\cM_{H}\ra \cM_{H_{\rm aux}^p}\otimes_{\ZZ_{(p)}}\QQ,\end{equation}
compatible with the map between the two PEL data, which is finite on the coarse moduli spaces.  
 
 \begin{prop}
 
 The normalization $\vec{\cM}_H$ of $\cM_{H_{\rm aux}^p}$ in $\cM_H$ is a normal flat algebraic stack over $\ZZ_{(p)}$ whose generic fibre is canonically isomorphic to $\cM_H$.  The normalization of $[\cM_{H_{\rm aux}^p}]$ in $[\cM_H]$ under the map of coarse moduli spaces  induced by (\ref{418}) is canonically isomorphic to $[\vec{\cM}_H]$, which is a quasi-projective normal flat scheme over $\ZZ_{(p)}$. Hence $\vec{\cM}_H\simeq [\vec{\cM}_H]$ is a scheme if $H$ is neat.
  
 \end{prop}
 
  From now on, we always assume $H$ is neat.

Let $\cM_H^{\rm tor}$ be the toroidal compactification of $\cM_H$ for a fixed admissible smooth rational polyhedral cone decomposition datum $\Sigma$ for $\cM_H$. 

\begin{prop}\label{cptchar0}

(1) There is an admissible smooth rational polyhedral cone decomposition datum $\Sigma_{\rm aux}$ for $\cM_{H_{\rm aux}^p}$ (hence the toroidal compactification $\cM_{H_{\rm aux}^p}^{\rm tor}$ of  $\cM_{H_{\rm aux}^p}$), which is compatible with $\Sigma$ in a natural way, and induces a canonical morphism  
\begin{equation}\label{4225}
\cM_H^{\rm tor}\lra \cM_{H_{\rm aux}^p}^{\rm tor}\otimes_{\ZZ_{(p)}}\QQ,
\end{equation}which is compatible with the stratifications on both sides (in particular, extending (\ref{418})) and the pull-back of universal objects. 

(2) Let $\cM_{H}^{\rm min}$ and $\cM_{H_{\rm aux}^p}^{\rm min}$ be the corresponding minimal compactifications.  Then the morphism (\ref{4225}) induces a natural morphism \[\cM_H^{\rm min}\lra \cM_{H_{\rm aux}^p}^{\rm min}\otimes_{\ZZ_{(p)}}\QQ,\]which is compatible with the stratifications on both sides.  The normalization $\vec{\cM}_H^{\rm min}$ of $\cM_{H_{\rm aux}^p}^{\rm min}$ in $\cM_H^{\rm min}$ is a projective normal flat scheme over $\ZZ_{(p)}$ whose generic fibre is canonically isomorphic to $\cM_H^{\rm min}$. It contains $\vec{\cM}_H$ as an open dense subscheme.  
 
(3) In the case that $\Sigma$ is projective, there is an integral model $\vec{\cM}_H^{\rm tor}$ for the toroidal compactication $\cM_H^{\rm tor}$, which is by construction the normalization of the blow-up of certain coherent ideal sheaf on $\vec{\cM}_H^{\rm min}$. It is a projective normal flat scheme over $\ZZ_{(p)}$, such that $\vec{\cM}_H^{\rm tor}\otimes_{\ZZ_{(p)}}\QQ\simeq \cM_H^{\rm tor}$ in a canonical way. If $H'\subset H$ is an open compact subgroup, then there is a canonical map $\vec{\cM}_{H'}^{\rm tor}\ra \vec{\cM}_H^{\rm tor}$,  compatible with the canonical map $\cM_{H'}\ra \cM_H$. 

 \end{prop}

 For the integral model $\vec{\cM}_{H^p}$ with prime to $p$ level, we have the following stronger result:

\begin{theo}[\cite{PR2}]\label{prmod}
The canonical map $\vec{\cM}_{H^p}\ra \cM_{H_{\rm aux}^p}$ is a closed embedding.

In particular, we have a moduli interpretation for $\vec{\cM}_{H^p}$, with PEL data as part of the moduli problem. 
\end{theo}

\begin{proof}

By Theorem 12.2 of \cite{PR2}, the flat scheme-theoretic image in $\vec{\cM}_{H_{\rm aux}^p}$ of the generic fibre $\cM_{H^p}$ is normal, hence is canoically isomorphic to $\vec{\cM}_{H^p}$. In another word, the integral model $\vec{\cM}_{H^p}$ defined above coincides with the canonical integral model of Pappas-Rapoport \cite{PR2}.

For the last claim, the reader is referred to Section 15 \cite{PR2} for more details.
 \end{proof}

\subsection{Ordinary loci and partial compactifications}

\subsubsection{Level structures prime to $p$} \label{leveln}

We recall certain results from Ch. 3 \cite{La12}. Let $S$ be a scheme over $\ZZ_{(p)}$. Let $A$ be an abelian scheme over $S$, equipped with polarization $\lambda$ and $\cO$-endomorphism $i$ as before. Let $H^p\subset G(\hat{\ZZ}^p)$ be an open compact. Let $N$ be a natural number prime to $p$  such that $H^p\supset U(N)$, the principal mod $N$ congruence subgroup. A principal level $N$ structure of $(A,\lambda, i)$ of type $(L_{\hat{\ZZ}^p},\psi)$ is the pair $(\alpha_N,\nu_N)$ defined as follows:

\begin{itemize}

\item $\alpha_N: L/NL\stackrel{\sim}{\ra}A[N]$ is an $\cO$-linear isomorphism of group schemes over $S$, such that 

(1) the symplectic pairing  \[L/NL \times  L/NL \lra \ZZ/N\ZZ\]  and the $\lambda$-Weil pairing \[A[N]\times A[N]\lra \mu_N\] induced by the polarization $\lambda$ are compatible for a chosen isomorphism of group schemes \[\nu_N:\ZZ/N\ZZ\stackrel{\sim}{\lra} \mu_N\] with respect to a fixed primitive $N$-th root of unity $\zeta_N$. 
 
 (2) $\alpha_H$ is symplectic liftable: there is a tower of finite \'etale surjections \[(S_M\twoheadrightarrow S_N=S)_{N\mid M, p\nmid M}\] and $\cO$-linear isomorphisms $\alpha_M: L/ML\stackrel{\sim}{\ra}A[M]$ with respect to an isomorphism $\nu_M:\ZZ/M\ZZ\stackrel{\sim}{\ra} \mu_M$ such that for any valid indices $M'|M''$  \[(\alpha_{M'},\nu_{M'})=(\alpha_{M''},\nu_{M''}) \text{ mod } M'.\] (This condition is required so that $\alpha_N$ lifts, at any geometric point $s$ of $S$, to an $\cO$-linear symplectic isomorphism between $L_{\hat{\ZZ}^p}$ and the Tate module of $A_s$.)

\end{itemize}

Consider all natural numbers $N$ such that $p\nmid N$ and $H^p\supset U(N)$. A level $H^p$ structure of $(A,\lambda, i)$ of type $(L_{\hat{\ZZ}^p},\psi)$ is a collection of $H^p/U(N)$-orbits of principal level $N$ structures  $(\alpha_N,\nu_N)$ for all $N$ as above.

\subsubsection{Ordinary level structures at $p$}

Let \[0=D^1\subset D^0\subset D^{-1}=L_{\ZZ_p}\] be a filtration of $\cO\otimes_{\ZZ}\ZZ_p$-modules, such that $\mathrm{Gr}_D^{-1}:=D^{-1}/D^0$ is torsion  free as a $\ZZ_p$-module, and under the pairing $\psi$ $D^0$ is totally isotropic and is its own annihilator.   Such a filtration determines a filtration \[0=D^{\vee,1}\subset D^{\vee,0}\subset D^{\vee,-1}=L_{\ZZ_p}^{\vee}\] on the dual lattice $L_{\ZZ_p}^{\vee}$. We have the natural map \[\phi_D^0:D^0\ra D^{\vee,0}\] whose reduction mod $p^n$ is denoted by $\phi_{D,p^n}^0$. 

Let $P_D\subset G_{\ZZ_p}$ be  the stabilizer of $D$. Let $M_D$ be the group over $\ZZ_p$, whose $R$-points, for any $\ZZ_p$-algebra $R$, are 
$(g,c)\in \GL_{\cO\otimes_{\ZZ}R}(\mathrm{Gr}_D\otimes_{\ZZ_p}R)\times \mathbf{G}_m(R)$ such that $\psi(gx,gy)=c\psi(x,y)$. We denote by $U_D$ the kernel of the natural morphism from $P_D$ to $M_D$. Now for any integer $n\in \ZZ_{\geq 0}$, we set 
\[U_{p,0}(p^n)=(G(\ZZ_p)\ra G(\ZZ/p^n\ZZ))^{-1}P_D(\ZZ/p^n\ZZ),\]
\[U_{p,1}^{\rm bal}(p^n)=(G(\ZZ_p)\ra G(\ZZ/p^n\ZZ))^{-1}U_D(\ZZ/p^n\ZZ).\]

Let $S$ be a scheme over $\ZZ$. Let $A$ be an abelian scheme together with a polarization $\lambda$ and an $\cO$-endomorphism $i$. An ordinary principal level $p^n$ structure of $(A,\lambda,i)$ of type $(L_{\ZZ_p},\psi, D)$ is the following data:

\begin{itemize}

\item  An $\cO$-linear homomorphism $\alpha_{p^n}^0: (D^0/p^nD^0)^{\rm mult}\ra A[p^n]$ of group schemes over $S$.

\item  An $\cO$-linear homomorphism $\alpha_{p^n}^{\vee,0}: (D^{\vee,0}/p^nD^{\vee,0})^{\rm mult}\ra A^{\vee}[p^n]$  of group schemes over $S$.

\item A section $\nu_{p^n}$ of $(\ZZ/p^n\ZZ)^{\times}\simeq \mathrm{Isom}_S(\mu_{p^n},\mu_{p^n})$ so that the homomorphism of multiplicative group schemes \[\nu_{p^n}\circ (\phi_{D,p^n}^0)^{\rm mult}: (D^0/p^nD^0)^{\rm mult}\ra (D^{\vee,0}/p^nD^{\vee,0})^{\rm mult}\] is compatible with $\lambda$ under $\alpha_{p^n}^0$ and $\alpha_{p^n}^{\vee,0}$, and such that the scheme theoretic images  $\mathrm{Im} (\alpha_{p^n}^0)$ and $\mathrm{Im}(\alpha_{p^n}^{\vee,0})$ kill each other under the $\lambda$-Weil pairing on $A[p^n]\times A^{\vee}[p^n]$.

\item The requirement that  $\alpha_{p^n}$ is symplectic liftable: there is a tower of quasi-finite \'etale surjections \[(S_{p^{n'}}\twoheadrightarrow S_{p^n}=S)_{n'\geq n}\] and triples $(\alpha_{p^{n'}}^0, \alpha_{p^{n'}}^{\vee,0}, \nu_{p^{n'}})$ as above such that for any $n''\geq n'$, \[(\alpha_{p^{n''}}^0, \alpha_{p^{n''}}^{\vee,0}, \nu_{p^{n''}}) \text{ mod } p^{n'}=(\alpha_{p^{n'}}^0, \alpha_{p^{n'}}^{\vee,0}, \nu_{p^{n'}}).\]  
\end{itemize}

Let $H_p\subset G(\ZZ_p)$ be an open compact subgroup such that $U_{p,1}^{\rm bal}(p^n)\subset H_p\subset U_{p,0}(p^n)$ for some integer $n\geq 0$. An ordinary level $H_p$ structure of $(A,\lambda,i)$ of type $(L_{\ZZ_p},\psi, D)$ is an $H_p/U_{p,1}^{\rm bal}(p^n)$-orbit of ordinary principal level $p^n$ structure of $(A,\lambda,i)$ of type $(L_{\ZZ_p},\psi, D)$.

\subsubsection{Integral models with ordinary level structures}

Let $H=H^pH_p\subset G(\hat{\ZZ})$ be an open compact subgroup such that $U_{p,1}^{\rm bal}(p^n)\subset H_p\subset U_{p,0}(p^n)$ for some integer $n\geq 0$. Let $\cM_H^{\rm ord, naive}$ be the functor that assigns to a $\ZZ_{(p)}$-scheme $S$ the isomorphism classes of the tuples $(A, i, \la, \alpha_{H^p}, \alpha_{H_p})$ as follows:
 
\begin{itemize}
\item $A$ is an abelian scheme over $S$ of relative dimension $dg$, equipped with an $\cO$-endomorphism $i: \cO\hookrightarrow \mathrm{End}_S(A)$.
\item $\la: A\ra A^{\vee}$ is a polarization.  

\item $\alpha_{H^p}$ is a level $H^p$ structure of type $(L_{\hat{\ZZ}^p}, \psi)$.

\item $\alpha_{H_p}$ is a level  $H_p$ structure of type $(L_{\ZZ_p}, \psi,D)$.
\end{itemize}

The functor $\cM_H^{\rm ord, naive}$ is represented by a scheme of finite type over $\ZZ_{(p)}$.

Let $r_{H}$ be the fixed nonnegative integer as in \cite{La12} which is determined by the PEL data and the filtration $D$ of $L_{\ZZ_p}$. One can check that, over any $\QQ[\zeta_{p^{r_H}}]$-scheme $S$, there is a natural assignment from the level $H$ structures of $(A,i,\lambda)_S$ to the pairs of   $H^p$-level structure of type $(L_{\hat{\ZZ}^p}, \psi)$ and $H_p$-level structure of type $(L_{\ZZ_p}, \psi)$, which is in fact injective. As a consequence, we have an open and closed immersion \[
\cM_H\otimes_{\QQ}\QQ[\zeta_{p^{r_H}}]\lra \cM_{H}^{\rm ord, naive}\otimes_{\ZZ}\QQ[\zeta_{p^{r_H}}],\]whose image is denoted by $\cM_H^{\rm ord}$, which is an open and closed subscheme of $\cM_{H}^{\rm ord, naive}\otimes_{\ZZ}\QQ[\zeta_{p^{r_H}}]$.

\begin{prop}

The normalization $\vec{\cM}_H^{\rm ord}$ of  $\cM_{H}^{\rm ord, naive}$ in $\cM_H^{\rm ord}$ under the natural morphism $\cM_H^{\rm ord}\ra \cM_{H}^{\rm ord, naive}$ is a scheme smooth quasi-projective  separated  of finite type over $\ZZ_{(p)}[\zeta_{p^{r_H}}]$, which is an open subscheme of $\vec{\cM}_H\otimes_{\ZZ_{(p)}}\ZZ_{(p)}[\zeta_{p^{r_H}}]$.

\end{prop}

 \subsubsection{Partial compactifications of ordinary loci}
 
 Keep the data as before.
 
 \begin{theo}\label{ordtor}
 
 There is a scheme $\vec{\cM}_H^{\rm ord,tor}$, quasi-projective smooth separated of finite type over $\ZZ_{(p)}[\zeta_{p^{r_H}}]$, containing $\vec{\cM}_H^{\rm ord}$ as an open dense  subscheme.  
 The universal tuple $(\cA,i,\lambda, \alpha_{H^p},\alpha_{H_p})$ on $\vec{\cM}_H^{\rm ord}$  extends to  $\vec{\cM}_H^{\rm ord,tor}$. The boundary $\vec{\cM}_H^{\rm ord,tor}\backslash \vec{\cM}_H^{\rm ord}$ is a relative Cartier divisor with normal crossing.

 \end{theo}

We have the Hodge line bundle $(\det\omega)_{\vec{\cM}_H^{\rm ord}}$ over $\vec{\cM}_H^{\rm ord}$, and its extension $(\det\omega)_{\vec{\cM}_H^{\rm ord,tor}}$ to $\vec{\cM}_H^{\rm ord,tor}$. Form 
\[\vec{\cM}_H^{\rm ord,min}=\mathrm{Proj}_{\oplus_{k\geq 0}}\Gamma(\vec{\cM}_H^{\rm ord,tor},(\det\omega)_{\vec{\cM}_H^{\rm ord,tor}}^k).\]This is in general not projective, as the partial toroidal compactification $\vec{\cM}_H^{\rm ord,tor}$ is  not proper. 

\begin{theo}\label{normal}

There exists a canonical proper morphism \[\vec{\cM}_H^{\rm ord,tor} \lra \vec{\cM}_H^{\rm ord,min}. \] The scheme $\vec{\cM}_H^{\rm ord,min}$ is quasi-projective  normal and flat over $\ZZ_{(p)}[\zeta_{p^{r_H}}]$ which contains $\vec{\cM}_H^{\rm ord}$ as an open dense subscheme.

 \end{theo}

 \begin{rk}
 
 For the moduli space $\vec{\cM}_{H^p}$ with prime to $p$ level, the integral model $\vec{\cM}_{H^p}^{\rm ord}$ is simply the ordinary locus of $\vec{\cM}_{H^p}\otimes_{\ZZ_{(p)}}\ZZ_{(p)}[\zeta_{p^{r_H}}]$. It then comes with a moduli interpretation, by Theorem \ref{prmod}.

 \end{rk}

\subsection{Hecke correspondences}

 \subsubsection{The double-coset Hecke algebra} \label{doublecoset}

 Let $q$ be a prime number and $v\mid q$ a place in $F$. For the completion $F_v$ of $F$ at the place $v$, we denote by $\cO_v$ the integer ring and fix a uniformizer $\varpi_v$.  We define the spherical Hecke algebra $\cH_v^{\rm sph}$ for $\GSp_{2g}(F_v)$  with coefficients in $\ZZ$ to be the algebra of $\mathbf{Z}$-valued functions on $\GSp_{2g}(F_v)$ that are bi-invariant under $\GSp_{2g}(\mathcal{O}_v)$. It is generated by the characteristic functions on the following double cosets:

  \[T_{v,1}=\GSp_{2g}(\cO_v)\begin{pmatrix}
                             I_g & \\
& \varpi_vI_g \end{pmatrix}\GSp_{2g}(\cO_v),\]

 \[T_{v,i}=\GSp_{2g}(\cO_v)\begin{pmatrix}
                             I_{g-i+1} & & & \\
& \varpi_vI_{i-1}& & \\& & \varpi_v^2I_{g-i+1}\\& & & \varpi_vI_{i-1} \end{pmatrix}\GSp_{2g}(\cO_v), \quad 2\leq i\leq g,\] 
\[
S_v = \varpi_v \GSp_{2g}(\mathcal{O}_v).
\]

\subsubsection{Weights and automorphic sheaves}\label{weight}

Through the end of this section, let $\cM$ denote  $\vec{\cM}_{H^p}$, $\vec{\cM}_{H^p}^{\rm ord}$, or $\cM_H$. 

We again denote the universal abelian scheme over $\cM$ by $\cA$, and denote by $\omega$ the pull-back of $\Omega_{\cA/\cM}^1$ along the unit section. We remark that $\omega$ is locally free over $\cO_{\cM}$, but is not locally free as an $\cO_{\cM}\otimes_{\ZZ}\cO_F$-module if $p$ is ramified in $F$,  for the integral models.

We only give the construction of  automorphic sheaves of $\cM_H$,  and those  of $\cM=\vec{\cM}_{H^p}$ when $p$ is a \emph{good} prime for the moduli, so that $\omega$ is locally free over $\cO_{\cM}\otimes_{\ZZ}\cO_F$. The latter is enough for the auxiliary moduli. The  automorphic  sheaves over $\vec{\cM}_{H^p}$ and $\vec{\cM}_{H^p}^{\rm ord}$ in the general case are then defined by restriction via the closed immersion in Theorem \ref{prmod} and the inclusion $\vec{\cM}_{H^p}^{\rm ord}\subset \vec{\cM}_{H^p}\otimes_{\ZZ_{(p)}}\ZZ_{(p)}[\zeta_{p^{r_H}}]$. We refer the reader to Chapter 8 \cite{La12} for more details, including the cases with level structure at $p$.

Let $T_{g/\mathbf{Z}}$ be the standard diagonal maximal torus of $\GSp_{2g/\mathbf{Z}}$. Put $G = \Res_{\mathbf{Z}}^{\mathcal{O}} \GSp_{2g/\mathbf{Z}}$ and $T=\Res_{\mathbf{Z}}^{\mathcal{O}} T_{g/\mathbf{Z}}$. Take the standrad Borel $B$ of $G$ with unipotent radical $U$ and identify $T=B/U$. 
Let $M$ be the Levi of  the standard Siegel parabolic of $G$. Then $M \supset T$.

Consider a character \[\kappa: T  \lra \mathbf{G}_m.\] We may regard $\kappa$ as a character of $B \cap M$ which is trivial on $U \cap M$.  The character $\kappa$ is called \emph{dominant} with respect to $B$, if the induced representation $\mathrm{Ind}_{B \cap M}^M\kappa^{-1}$ inside the rational functions of  the scheme $(M/U \cap M)$ is non-zero. The Bruhat-Tits decomposition shows that the subspace $(\mathrm{Ind}_{B\cap M}^M\kappa^{-1})^{U \cap M}$ is one dimensional, and $T$ acts on a generator by $-w_0\kappa$, where $w_0$ is the longest element in the Weyl group (with respect to $T$). The $M$-translation of the generator generates a sub-representaton \[\rho_{\kappa}^{*}\subset \mathrm{Ind}_{B \cap M}^M\kappa^{-1},\] where an element $m$ in the standard Levi $M$ acts as $m\cdot f(x)=f(m^{-1}x)$. The $R$-dual $\rho_{\kappa}$ of $\rho_{\kappa}^{*}$  is called the \emph{rational representation of highest weight} $\kappa$, which has the universal property that for any $M$-module $X$, 
\[\Hom_M(\rho_{\kappa},X)\simeq \Hom_M(X^*, \rho_{\kappa}^{*})\simeq \Hom_B(X^*,-\kappa)\simeq \Hom_B(\kappa, X).\]

We define the \emph{automorphic sheaf} of weight $\kappa$ on $\cM$ to be
 \[\omega^{\kappa}=\pi_*\cO_{M(\omega)/U\cap M(\omega)}[\kappa]\] for $\pi: M(\omega)/U\cap M(\omega)\ra \cM$.  
The sections of $\omega^{\kappa}$ are functions on $\mathrm{Isom}_{\cM}(\cO_{\cM}^{dg},\omega)$ such that \[f(\phi m)=\rho_{\kappa}^*(m)f(\phi),\quad \forall \phi\in \mathrm{Isom}_{\cM}(\cO_{\cM}^{dg},\omega),m\in M.\] 

The construction above then provides the automorphic sheaves on $\cM_H$ and the ones on the auxiliary moduli in the integral cases, and then those on the integral models $\cM=\vec{\cM}_{H^p}$,  $\vec{\cM}_{H^p}^{\rm ord}$ without the assumption that $p$ is a good prime for the moduli problems.  We always denote the automorphic sheaves over $\cM$ by the same symbol $\omega^{\kappa}$.
 
By the results of Chapter 8 \cite{La12}, the automorphic sheaf   $\omega^{\kappa}$ extends from the moduli schemes to the total (resp. partial) compactifications in a canonical way, which is compatible with the restrictions to the ordinary loci of the total objects.

\subsubsection{Geometric correspondences}\label{corr}
As in the previous section, we may assume $p$ is a good prime for the moduli $\cM$. 

Let $\fa$ be an ideal of $\cO$.   Let $\cM^{\fa}$ be the moduli stack of isogenies between
objects in $\cM$, that is, the algebraic stack representing the functor $\cM^{\fa}$ which assigns to any  base scheme $S$ over $\QQ$ (resp. $\ZZ_{(p)}$,  resp. $\ZZ_{(p)}[\zeta_{p^{r_H}}]$) the
category in groupoids in which an object is an isogeny \[f: A\ra B\] between
two polarized abelian schemes with endomorphisms and  level structures $(A,i_A,\lambda_A)$ and $(B,i_B,\lambda_B)$, whose kernel is (\'etale locally) $\cO$-linearly isomorphic to $(\cO/\fa\cO)^g$ and intersects with (the image of) the level structure only along the unit section, is compatible with the $\cO$-endomorphisms, and  respects the polarizations on both sides. 

Here we obtain the representability of the functor $\cM^{\fa}$ by the use of the fact that  $\cM$ is representable and by the theory of Hilbert schemes (cf. P. 251 \cite{FC}). In particular, since  $H$ is assumed to be neat, the functor $\cM^{\fa}$ is represented by a quasi-projective scheme over $\QQ$ (resp. $\ZZ_{(p)}$, resp. $\ZZ_{(p)}[\zeta_{p^{r_H}}]$), which is denoted by the same symbol, as usual. The universal isogeny over $\cM^{\fa}$  is denoted by $\mathcal{I}^{\fa}$. Assigning such an isogeny to its source (resp. target), we have two
natural projections

\[\cM\stackrel{\pi_{1,\fa}}{\longleftarrow}\cM^{\fa}\stackrel{\pi_{2,\fa}}{\lra}\cM,\]  whose restrictions to any connected component $Z$
of $\cM^{\fa}$ are proper, by the valuative criterion. 

In the case that  $p$ is invertible in the base scheme $S$, the two projections \[\pi_{i,(p)}:\cM^{(p)}\ra \cM, \quad i=1,2\] are finite \'etale. 
In this case, for  $v\mid q$ a prime ideal in $\cO$, one has the bijection between the connected components of $\cM^{v}$ and the  double cosets $\gamma_v$ in the spherical Hecke algebra $\cH_v^{\rm sph}$.  Denote the corresponding connected component of  $\cM^{v}$ by $\cM^{\gamma_v}$, over which the universal isogeny is said to be of type $\gamma_{v}$.   We have the two projections \[\pi_{i,\gamma_v}:\cM^{\gamma_{v}}\ra \cM, \quad i=1,2,\] of type $\gamma_{v}$.


For $\cM=\vec{\cM}_{H^p}^{\rm ord}, \vec{\cM}_{H^p}$ over a scheme $S$ in characteristic $p$,  and $v|p$ a prime ideal in $\cO$, we again have the   connected component of  $\cM^{\gamma_v}$ and the two projections \[\pi_{i,\gamma_v}:\cM^{\gamma_{v}}\ra \cM, \quad i=1,2,\] of type $\gamma_{v}$. (We refer the reader to Ch. VII \cite{FC} for details on the facts above.)

In the two cases above, consider the commutative diagram
\[\CD
  \mathcal{I}^{\fa} @>  >> \mathcal{A} \\
  @V f_{\fa} VV @V f VV  \\
  Z=\cM^{\fa} @>\pi_{i,\fa}>> \cM
\endCD \]

Over the base $S$, we have a natural map of
$\cO\otimes_{\ZZ}\cO_{Z}^{\times}$-torsors
\[\pi_{2,\fa}^*\omega=
\pi_{2,\fa}^*(f_*\Omega_{\mathcal{A}/\cM}^1)\ra
f_{\fa*}\Omega_{\mathcal{I}^{\fa}/Z}^1\stackrel{\sim}{\ra}f_{\fa*}(\pi_{1,\fa}^*\Omega_{\mathcal{A}/\cM}^1)\stackrel{\sim}{\ra}
\pi_{1,\fa}^*\omega, \]hence the induced map
\[\theta: \pi_{2,\fa}^*\omega^{\kappa}\lra \pi_{1,\fa}^*\omega^{\kappa}.\] Applying $\pi_{1,\fa *}$ and
composing with the trace map \[\mathrm{tr}:
\pi_{1,\fa *}\pi_{1,\fa}^*\omega^{\kappa}\ra
\omega^{\kappa},\] we obtain the map
\[\pi_{1,\fa *}\pi_{2,\fa}^*\omega^{\kappa}\stackrel{\pi_{1,\fa *}\theta}{\longrightarrow}
\pi_{1,\fa *}\pi_{1,\fa}^*\omega^{\kappa}\stackrel{\rm tr}{\lra}
\omega^{\kappa}.\]Composing the natural map
\[\omega^{\kappa}\ra
\pi_{1,\fa *}\pi_{2,\fa}^*\omega^{\kappa}\] with the one above and
taking global sections, one gets the desired endomorphism
\[T_{\fa}: H^0(\cM_R,\omega^{\kappa})\ra
 H^0(\cM_R,\omega^{\kappa}),\] which will be denoted by $U_{(p)}$ in the case $\fa=(p)$. We remark that the Hecke operator  $U_{(p)}$ corresponds to the product of the double cosets $T_{v,1}, v|p$.
 
We have the same construction for $Z=\cM^{\gamma_v}$, the connected component of $\cM^{v}$ of type $\gamma_v$. In these cases, the Hecke operators  corresponding to the double cosets $T_{v,i}$ (resp. $S_v$) will be denoted by $T_{v,i}$  (resp. $S_v$) again.

\subsection{Hasse invariants  and liftings}

 \subsubsection{Hasse invariants on abelian schemes in characteristic $p$}
 
 Let  $A$ be an abelian scheme over $S$, a scheme in characteristic $p$. We have $A^{(p)}$, the pull-back of  $A\ra S$ via the absolute Frobenius $\mathrm{Frob}_S$ on $S$, and $V_{A/S}: A^{(p)}\ra A$,  the Verschiebung isogeny. The latter then induces the map \[\mathcal{C}_{A/S}: \Omega_{A/S}^1\ra \Omega_{A^{(p)}/S}^1\simeq \mathrm{Frob}_S^*\Omega_{A/S}^1 ,\] whose highest exterior power gives
 \[\mathcal{C}_{A/S}: \det\omega_{A/S} \ra (\det \omega_{A/S})^p ,\]hence a section $h\in H^0(A,(\det\omega_{A/S})^{\otimes (p-1)})$.   Applying this to the universal abelian scheme on the special fibre of $\vec{\cM}_{H^p}$,  we then have a global section \[h\in H^0((\vec{\cM}_{H^p})_{\mathbf{F}_p},(\det\omega)^{\otimes p-1}),\] which is known as the Hasse invariant of the moduli space. We have its extensions to $\vec{\cM}_{H^p}^{\rm tor}$ and $\vec{\cM}_{H^p}^{\rm min}$, and   denote them by $h$ again. 
 The reader is referred  to Section 6.3 \cite{La12} for more details.
 
 \begin{rk}\label{vanish}

For  $A/S$ an abelian scheme of dimension $n$,  the Hasse invariant $h(A)$ is non-vanishing if and only if $A$ is ordinary, which means that $A[p]$ has $p^n$ elements at every geometric point of $S$.     
\end{rk}

  \begin{lemma}\label{func}
 
Recall the notation from Section \ref{corr}. The natural map of sheaves on  $(\vec{\cM}_{H^p}^{\fa})_{\mathbf{F}_p}$ (resp. $(\vec{\cM}_{H^p}^{\gamma_v})_{\mathbf{F}_p}$)  
 \[\theta: \pi_{2}^*(\det\omega)^{\otimes p-1}\ra \pi_{1}^*(\det\omega)^{\otimes p-1}\] satisfies
 \[\theta(\pi_{2}^*h)=\pi_{1}^*h.\]   
 
 \end{lemma}
 
 \begin{proof}
 
 This is by the functoriality of the Cartier operator. 
 
 Let $R$ be an $\mathbf{F}_p$-algebra. Note that \[\theta(\pi_{2}^*h)(A\ra B, \mathbf{v}, \mathbf{v}')=h(B, \mathbf{v}'), \quad (\pi_{1}^*h)(A\ra B, \mathbf{v}, \mathbf{v}')=h(A,\mathbf{v}),\] where $\mathbf{v}$ (resp. $\mathbf{v}'$) is a chosen basis of $H^0(A, \Omega_{A/R}^1)$ (resp. $H^0(B, \Omega_{B/R}^1)$) so that $\pi_2^*\mathbf{v}=\mathbf{v}'$. On the other hand, we have 
 \[h(B, \mathbf{v}')\mathrm{Frob}_R^*(\det\mathbf{v})=h(B, \mathbf{v}')\mathrm{Frob}_R^*(\pi_{2}^*(\det\mathbf{v}'))\]\[=(\pi_{2}^{(p)})^*\mathcal{C}_{B/R}(\det\mathbf{v}')=\mathcal{C}_{A/R}(\pi_{2}^*(\det\mathbf{v}')),\]
 where the last equality is obtained by the \'etaleness of the projection $\pi_{2}$.
 
 Now the result follows, as by definition \[\mathcal{C}_{A/R}(\pi_{2}^*(\det\mathbf{v}'))=\mathcal{C}_{A/R}(\det\mathbf{v})=h(A,\mathbf{v})\mathrm{Frob}_R^*(\det\mathbf{v}).\]\end{proof}
 
 \subsubsection{Lifting Hasse invariants to characteristic zero}

Consider the Hasse invariant \[h\in H^0((\vec{\cM}_{H^p}^{\rm  min})_{\mathbf{F}_p},   (\det\omega)^{\otimes p-1}).\]  Since the line bundle $(\det\omega)^{\otimes p-1}$ on $(\vec{\cM}_{H^p}^{\rm min})_{\mathbf{F}_p}$ is ample, the line bundle $(\det\omega)^{\otimes k(p-1)}$ is very ample for sufficiently large integer $k$.   
 \begin{prop}\label{lifttochar0}
 For $k\in \ZZ_{\geq 0}$ big enough, the section $h^k\in H^0((\vec{\cM}_{H^p}^{\rm min})_{\mathbf{F}_p}, (\det\omega)^{\otimes (p-1)k})$ lifts to a section \[\widetilde{h^k}\in H^0(\vec{\cM}_{H^p}^{\rm min},   (\det\omega)^{\otimes (p-1)k}).\]
 \end{prop}
\begin{proof} 

It is standard to show by Serre vanishing that \[H^1(\vec{\cM}_{H^p}^{\rm  min},  (\det\omega)^{\otimes (p-1)k})=0\] when $k$ is sufficiently large, which in turns gives the surjectivity of
\[H^0(\vec{\cM}_{H^p}^{\rm min},  (\det\omega)^{\otimes (p-1)k})\ra H^0((\vec{\cM}_{H^p}^{\rm min})_{\mathbf{F}_p},  (\det\omega)^{\otimes (p-1)k}).\]\end{proof}

From now on, we fix such a lift as in Proposition \ref{lifttochar0} \[E:=\widetilde{h^{k_0}}\] such that $k_0\gg 0$ and $p\nmid k_0$.

 \section{Analytification of Siegel-Hilbert moduli schemes}
 
\subsection{Preliminary}

We recall certain definitions and results from \cite{KL} and \cite{Lu}, which will be applied to the rigid analytifications of the Siegel-Hilbert moduli schemes, as well as their formal models and the automorphic sheaves.
 
 \subsubsection{Relative compactness}
 
Let $K$ be a finite extension of $\QQ_p$ and $\cO_K$ the ring of integers. For $\fX$ a formal scheme over $\Spf \cO_K$, we denote by $\fX^{\rm rig}$ the rigid analytic space associated to it, and by $\fX_0$
its special fibre. For $\fU\subset \fX$ an admissible open, we let $]\fU_0[$ denote the tube of $\fU$, i.e. the pre-image in $\fX^{\rm rig}$ of the $\fU_0$ under the natural specialization $\fX^{\rm rig}\ra \fX_0$, which is surjective. If $f: X\ra Y$ is morphism of rigid spaces, we call a morphism of $\varpi$-adic $\cO_K$-flat formal schemes $\mathfrak{f}: \fX\ra \fY$ a formal model of $f$, if $f$ is the rigidification of $\mathfrak{f}$.  

\begin{defn}\label{relcpt}
Let $f: X\ra V$ be a quasi-compact morphism of rigid analytic spaces, and $U\subset X$ a quasi-compact (relative to $X$) admissible open. We say $U$ is \emph{relatively compact in $X$ over $V$}, denoted by \[U\Subset_V X,\] if there exists an admissible covering of quasi-compact subsets $\{V_i\}$ of $V$ such that locally (over each $V_i$) there is a closed $V$-immersion of $X$ into an $n$-dimensional unit ball $\mathbf{D}_V^n(1)$ under which $U$ maps into a ball $\mathbf{D}_V^n(\epsilon)$ for some $\epsilon<1$. 

If $V=\Sp K$, we simply write $U\Subset_V X$ as $U\Subset X$.
 \end{defn}
 
The notion of relative compactness in Definition \ref{relcpt} is independent of the choice of covering $\{V_i\}$, essentially by Raynaud's theorem that the category of quasi-compact rigid spaces is equivalent to that of quasi-compact admissible formal schemes localized by admissible blow-ups. 

  \begin{lemma}[2.1.8 \cite{KL}] \label{218} 
 
 Let $i: Y\hookrightarrow Y'$ and $j: X\hookrightarrow X'$ be admissible open inclusions, all of which are quasi-compact rigid spaces over the quasi-compact rigid space $V$, and $f: Y'\ra X'$ be a proper morphism. Suppose we have the following Cartesian diagram
 
\[\begin{CD}
Y @>f>> X\\
@VViV @ VVjV\\
Y'@>f>> X'
\end{CD}\] and suppose $X\Subset_VX'$. Then $Y\Subset_VY'$.

 \end{lemma}

 \subsubsection{Overconvergence}\label{overconvergence}

 Let $\bar{X}$ be a quasi-compact rigid space over $K$, with a formal model  $\bar{\fX}$. Let $D\subset \bar{\fX}_0$ be a Cartier divisor. Choose a finite covering $\{\fU_i\}_{i=1,\cdots,n}$ of $\bar{\fX}_0$ so that for any $i$ the ideal of  $D|_{\fU_i}$ is generated by a single section $h_i\in \cO_{\bar{\fX}_0}$. Choose for each $h_i$ a lifting $\tilde{h}_i\in\Gamma(\bar{\fX}, \cO_{\bar{\fX}})$. For any $r\in (|p|^{1/e},1]$, define
 \[\bar{X}(r)=\bigcup_{1\leq i\leq n}\{x\in \bar{X}||\tilde{h}_i|\geq r\},\]which is independent of the choice of $\tilde{h}_i$.

 \begin{prop}[2.3.2 \cite{KL}]\label{23}  If $\bar{X}(1)\Subset_{\bar{X}} X'$ for $X'$ a quasi-compact admissible open of $\bar{X}$, then for any $r$ close enough to $1$ we have $\bar{X}(r)\subset X'$.

 \end{prop}

 \subsection{Lifting Frobenius}

Let $K$ be a finite extension of $\QQ_p$.  Let $X$ be a scheme which is locally of finite type over $\cO_K$, and $\fX$ the formal completion of $X$ along its special fiber. Denote by $\fX^{\rm rig}$ the rigid fibre of $\fX$ via Raynaud's functor. There is a natural morphism between the rigid fibre $\fX^{\rm rig}$ and the analytification $X^{\rm an}:=(X\otimes_{\cO_K} K)^{\rm an}$ of $X$:\[\fX^{\rm rig}\ra X^{\rm an} ,\] which is an isomorphism if $X$ is proper.

\subsubsection{Passing to the rigid analytic picture}

We glue $\vec{\cM}_{H^p}\otimes_{\ZZ_{(p)}}\ZZ_{(p)}[\zeta_{p^{r_H}}]$ with $\vec{\cM}_{H^p}^{\rm ord,tor}$ (resp. $\vec{\cM}_{H^p}^{\rm ord,min}$) and denote the resulting scheme by $\bar{\cM}_{H^p}$ (resp. $\cM_{H^p}^{*}$).

Let $\fM_{H^p}^*$ (resp. $\bar{\fM}_{H^p}$) be  the   formal completion of $\cM_{H^p}^*$  (resp. $\bar{\cM}_{H^p}$)  along the special fiber over $p$.  
The same procedure of completion along the special fibre gives the universal semi-abelian schemes  \[\fA^{*}\ra \fM_{H^p}^{*},\quad\bar{\fA}\ra \bar{\fM}_{H^p}.\] Combining with rigidification, we get the universal semi-abelian scheme \[\fA^{*,\rm rig}\ra \fM_{H^p}^{*,\rm rig},\quad\bar{\fA}^{\rm rig}\ra \bar{\fM}_{H^p}^{\rm rig}.\]
Take $\fM_{H^p}$   to be the open formal subscheme of $\fM_{H^p}^{*}$   whose points are in  $\vec{\cM}_{H^p}$, which is hence equipped with the universal objects \[\fA= \fA^{*}|_{\fM_{H^p}}\ra \fM_{H^p}.\] Set $\fM_{H^p}^{\rm rig}$ to be the open subspace of $\bar{\fM}_{H^p}^{\rm rig}$ lying over $\vec{\cM}_{H^p}^{\rm an}$, which then comes with  the restriction \[\fA^{\rm rig}=\bar{\fA}^{\rm rig}|_{\fM_{H^p}^{\rm rig}}\ra \fM_{H^p}^{\rm rig}.\]

We have the analogous construction and notation for the ordinary locus, by adding the superscript $\rm ord$. 

\subsubsection{Canonical subgroup and Frobenius}

The complement $D$ of the ordinary locus of the special fibre $(\bar{\cM}_{H^p})_0$ of $\bar{\cM}_{H^p}$ is a Cartier divisor, which is the vanishing locus of the Hasse invariant $h$,  by Theorems \ref{prmod}, \ref{ordtor}. We apply the construction in Section \ref{overconvergence} to $\bar{X}=\bar{\cM}_{H^p}^{\rm rig}$, and then have for $r\in (|p|^{1/e},1]$ the quasi-compact admissible opens $\bar{\fM}_{H^p}^{\rm rig}(r)\subset \bar{\cM}_{H^p}^{\rm rig}$. In particular, we see by definition that $\bar{\fM}_{H^p}^{\rm rig}(1)=\bar{\fM}_{H^p}^{\rm ord, rig}$. 

It is elementary to check by Definition \ref{relcpt} and Corollary 5.11 \cite{Lu} that, for $r,s\in (|p|^{1/e},1)$ with $s<r$,  \begin{equation}\label{luk}\bar{\fM}_{H^p}^{\rm rig}(r)\Subset \bar{\fM}_{H^p}^{\rm rig}(s)\end{equation} and \[\bar{\fM}_{H^p}^{\rm ord, rig}\Subset \bar{\fM}_{H^p}^{\rm rig}(r).\]

One has the following result from \cite{Fa}. (cf.  4.1.3 \cite{AIP} for extension to semi-abelian schemes.)

\begin{theo}[Theorem 6, \cite{Fa}]\label{822}

For each $n\in \ZZ_{\geq 1}$ and $r$ sufficiently close to $1$ (depending on $n$), there is a canonical subgroup of level $n$ $\cH_{n}(r)\subset \bar{\fA}^{\rm rig}[p^n]$  the $p$-divisible group $G=\bar{\fA}^{\rm rig}[p^{\infty}]$ 
over $\bar{\fM}_{H^p}^{\rm rig}(r)$,  which is locally free of rank $p^{dg}$ over $\cO_F$, and whose restriction to the ordinary locus $\fM_{H^p}^{\rm ord, rig}$ is the multiplicative subgroup $\fA^{\rm rig}[p^n]^{\rm mult}\subset \fA^{\rm rig}[p^n]$.

Moreover, the level-$1$ canonical subgroup $\cH_{1}(r)$ is the kernel of the Frobenius on $G$, and for any $1\leq m\leq n$, $\cH_{n}(r)/\cH_{m}(r)$ is the canonical subgroup of $\bar{\fA}^{\rm rig}[p^n]/\cH_{m}(r)$ of level $n-m$. 

\end{theo}

The multiplicative subgroup $\fA[p^n]^{\rm mult}\subset \fA[p^n]$ is a finite flat group scheme of order $p^{ndg}$.  One has, by the proof of 1.11.6 \cite{Ka78}, that $\fA/\fA[p^n]^{\rm mult}$   gives an element in  $\fM_{H^p}^{\rm ord}$. We thus have a canonical map 

\[\phi^n: \fM_{H^p}^{\rm ord}\ra \fM_{H^p}^{\rm ord}, \quad \fA\mapsto \fA/\fA[p^n]^{\rm mult}. \]

\begin{prop}\label{31} The morphism $\phi$ induces the following morphism  if $r$ is close enough to $1$
,  which is finite flat of degree $p^{ndg}$:

\[\overline{\phi^n}(r): \bar{\fM}_{H^p}^{\rm rig}(r)\ra \bar{\fM}_{H^p}^{\rm rig}(r^{p^n}).\]

\end{prop}

\begin{proof}

One has the map $\overline{\phi^n}: \bar{\fM}_{H^p}^{\rm ord}\ra \bar{\fM}_{H^p}^{\rm ord}$ defined by $\bar{\fA}\mapsto \bar{\fA}/\bar{\fA}[p^n]^{\rm mult}$, and then the map $\overline{\phi^n}^{\rm rig}: \bar{\fM}_{H^p}^{\rm ord, rig}\ra \bar{\fM}_{H^p}^{\rm ord, rig}$ induced by the first one on the rigid fibre. 

The claim then follows from  the argument in the proof of 3.1.7 \cite{KL}, together with the observation on abelian schemes  in 1.11.4 \cite{Ka78}.
\end{proof}

\begin{coro}\label{ff}

For $r$ sufficiently close to $1$, the sheaf $\bar{\fA}^{\rm rig}[p^n]/\cH_{n}(r)$ is finite flat. 

\end{coro}

\begin{proof}
This is by Proposition \ref{31}.
\end{proof}
 
\subsection{The formal Igusa tower} \label{igusa}

Following the idea of Hida (see e. g. \cite{Hi}),  we define 
$\bar{\fM}_{H^pp^n}^{\rm ur,ord}$ to be the Galois cover of $\bar{\fM}_{H^p}^{\rm ord}$ trivializing the \'etale sheaf $\bar{\fA}[p^n]/\bar{\fA}[p^n]^{\rm mult}$. 
The pre-image of $\fM_{H^p}^{\rm ord}$ under the covering map is written as $\fM_{H^pp^n}^{\rm ur,ord}$.   
 We then have a proper map \[\bar{\fM}_{H^pp^n}^{\rm ur,ord}\ra \bar{\fM}_{H^p}^{\rm ord}\ra \fM_{H^p}^{*,\rm ord},\]for which the Stein factorization is written as $\fM_{H^pp^n}^{*,\rm ur, ord}$. 

Similarly as before, we have the universal semi-abelian scheme  \[\bar{\fA}_n\ra \bar{\fM}_{H^pp^n}^{\rm ur,ord}\] and  the universal abelian scheme \[\fA_n\ra \fM_{H^pp^n}^{\rm ur, ord}\] by restriction. Consider the associated rigid space $\bar{\fM}_{H^pp^n}^{\rm ur, ord,rig}$ and  the open subspace $\fM_{H^pp^n}^{\rm ur, ord, rig}$ being the intersection of $\bar{\fM}_{H^pp^n}^{\rm ur,ord,rig}$ with $\cM_{H}^{\rm ord,\rm an}=\vec{\cM}_{H}^{\rm ord,\rm an}$, where for $H$, $H_p$ has been chosen to provide the same level structure at $p$.  We then have the finite \'etale map \[\mathrm{Ig}_n: \bar{\fM}_{H^pp^n}^{\rm ur, ord,rig}\ra \bar{\fM}_{H^p}^{\rm ord, rig}, \] which is a Galois cover of $\bar{\fM}_{H^p}^{\rm ord, rig}$.

For $r$ close enough to $1$, we set \[\bar{\fM}_{H^pp^n}^{\rm ur, rig}(r)\rightarrow \bar{\fM}_{H^p}^{\rm rig}(r)\] to be the Galois cover of $\bar{\fM}_{H^p}^{\rm rig}(r)$ trivializing the finite flat sheaf $\bar{\fA}^{\rm rig}[p^n]/\cH_{n}(r)$ (see Corollary \ref{ff}). In particular, the ordinary locus $\fM_{H^pp^n}^{\rm ur,ord,rig}:=\bar{\fM}_{H^pp^n}^{\rm ur,rig}(1)$ is the rigid fibre of the formal Igusa tower $\fM_{H^pp^n}^{\rm ur,ord}$ constructed above, which justifies the notation. Set \[\fM_{H^pp^n}^{\rm ur, rig}(r)=\bar{\fM}_{H^pp^n}^{\rm ur, rig}(r)\cap \cM_{H}^{\rm an},\] for $r$ close enough to $1$. For $s<r$ with $s$ close enough to $1$, we have \begin{equation}\label{forcontinuity}
\bar{\fM}_{H^pp^n}^{\rm ur, rig}(r)\Subset \bar{\fM}_{H^pp^n}^{\rm ur, rig}(s),\end{equation} by Lemma \ref{218} and (\ref{luk}).

Again by the proof of 1.11.6 \cite{Ka78}, we have a well-defined map \[\phi_{n}: \fM_{H^pp^n}^{\rm ur,ord}\ra \fM_{H^pp^n}^{\rm ur,ord}, \quad \fA_n\mapsto \fA_n/\fA_n[p]^{\rm mult}. \]
 sitting in the following commutative diagram

\[\begin{CD}
\fM_{H^pp^n}^{\rm ur,ord} @>\phi_{n}>> \fM_{H^pp^n}^{\rm ur,ord}\\
@VV\mathrm{Ig}_nV @ VV\mathrm{Ig}_nV\\
\fM_{H^p}@>\phi^1>> \fM_{H^p}
\end{CD}\]

\begin{prop}\label{326}

 The map $\phi_{n}$ induces, for $r$ close enough to $1$, the following morphism which is finite flat of degree $p^{dg}$:
 
\[\bar{\phi}_{n}(r): \bar{\fM}_{H^pp^n}^{\rm ur, rig}(r)\ra \bar{\fM}_{H^pp^n}^{\rm ur, rig}(r^p).\]

\end{prop}

\begin{proof}

 This follows from  Propositions \ref{31} and 2.2.1 \cite{KL}.\end{proof}

 \section{Overconvergence of Siegel-Hilbert modular forms}
 
 \subsection{P-adic Banach spaces}
 
 Set $\bar{\mathcal{Z}}^{\rm ord,rig}\subset \bar{\fM}_{H^pp^n}^{\rm ur, ord,rig}$ to be the tube of $(\bar{\fM}_{H^pp^n}^{\rm ur,ord})_0\backslash (\fM_{H^pp^n}^{\rm ur,ord})_0$, and set $\mathcal{Z}^{\rm ord,rig}$ to be its intersection with $\fM_{H^pp^n}^{\rm ur, ord,rig}$. 
 
 \begin{lemma}[K\"{o}cher principle]\label{koecher}
 
 For $r< 1$ which is sufficiently close to $1$, we have the natural isomorphisms
 
 \[H^0(\bar{\fM}_{H^pp^n}^{\rm ur, rig}(r), \omega^{\kappa})\stackrel{\sim}{\ra}H^0(\fM_{H^pp^n}^{\rm ur, rig}(r), \omega^{\kappa})\stackrel{\sim}{\ra}H^0(\fM_{H^pp^n}^{\rm ur, rig}(r)\backslash\mathcal{Z}^{\rm ord,rig}, \omega^{\kappa}).\]

 \end{lemma}
\begin{proof}
 It suffices to show that, on formal affine open subsets, the natural map 
\[H^0(\bar{\fM}_{H^pp^n}^{\rm ur,ord}, \omega^{\kappa})\ra H^0(\fM_{H^pp^n}^{\rm ur,ord}, \omega^{\kappa})\]is an isomorphism. Furthermore, by the existence of the Galois covering of $\fM_{H^p}^{\rm ord}$ by $\fM_{H^pp^n}^{\rm ur,ord}$, we just need to show the isomorphism for $n=0$, which then reduces to the K\"{o}cher principle proved in \cite{La12}. \end{proof}
 
Recall we have the connected component $\cM_{H}^{\gamma_v,\rm an}$ of  $\cM_{H}^{v,\rm an}$ for any prime ideal $v\subset \cO$ and double coset $\gamma_v$. Define $\fM_{H^pp^n}^{\gamma_v,\rm ur,rig}(r)$ by the Cartesian diagram 

\[\begin{CD}
 \fM_{H^pp^n}^{\gamma_v,\rm
ur,rig}(r) @>>>  \cM_{H}^{\gamma_v,\rm  an}\\
@VVV @ VV\pi_{1,\gamma_v}V\\
 \fM_{H^pp^n}^{\rm
ur,rig}(r)@>>>  \cM_{H}^{\rm an}
\end{CD}\] 
 
Note that the left vertical map is finite \'etale, being the base change of the finite \'etale  map $\pi_{1,\gamma_v}$. We also denote it by $\pi_{1,\gamma_v}$ and the other projection by $\pi_{2,\gamma_v}$.

 \begin{prop}\label{norm}

Let $r< 1$ be sufficiently close to $1$. We have that $\cO(\fM_{H^pp^n}^{\gamma_v,\rm
ur,rig}(r))$ is a $p$-adic Banach space with respect to the norm
\[|f|_r:=\mathrm{sup}_{x\in \fM_{H^pp^n}^{\gamma_v,\rm ur,rig}(r)}|f(x)|\]
for a function $f\in \cO(\fM_{H^pp^n}^{\gamma_v,\rm ur,rig}(r))$. This is the same as the norm \[|f|_r^{\circ}:=\mathrm{sup}_{x\in \fM_{H^pp^n}^{\gamma_v,\rm ur,rig}(r)\backslash\pi_{1,\gamma_v}^{-1}(\mathcal{Z}^{\rm ord,rig})}|f(x)|.\] 
\end{prop}
\begin{proof}

We use the argument in the proof of 4.1.6 \cite{KL}.

First note that $|f|_r^{\circ}$ is a well-defined norm on the $p$-adic Banach space \[\cO(\fM_{H^pp^n}^{\gamma_v,\rm ur, rig}(r)\backslash\pi_{1,\gamma_v}^{-1}(\mathcal{Z}^{\rm ord,rig})),\] with the latter space being finite over $\bar{\fM}_{H^pp^n}^{\rm ur, rig}(r)\backslash \bar{\mathcal{Z}}^{\rm ord,rig}$ and hence quasi-compact. (If $\cF$ is a coherent sheaf on a quasi-compact rigid space $X$, then $\cF(X)$ is a $p$-adic Banach space.)

Then we only need to show that
\numberwithin{equation}{subsection} \begin{equation}\label{twonorms}
|f|_r=|f|_r^{\circ}, \quad \forall f\in \cO(\fM_{H^pp^n}^{\gamma_v,\rm ur, rig}(r)),\end{equation} which will then realize  \[\cO(\fM_{H^pp^n}^{\gamma_v,\rm ur,rig}(r))\hookrightarrow \cO(\fM_{H^pp^n}^{\gamma_v,\rm ur, rig}(r)\backslash\pi_{1,\gamma_v}^{-1}(\mathcal{Z}^{\rm ord,rig}))\] as a closed subspace. For this, recall we have the  
K\"{o}cher principle and the fact that the natural projection \[
\fM_{H^pp^n}^{\gamma_v,\rm ur,ord,rig}\stackrel{\pi_{1,\gamma_v}}{\lra} \fM_{H^pp^n}^{\rm ur,ord,rig}\stackrel{\mathrm{Ig}_n}{\lra}  \fM_{H^p}^{\rm rig}\] is finite
\'{e}tale. The former is Lemma \ref{koecher}, and the latter holds
because both $\pi_{1,\gamma_v}$ and $\mathrm{Ig}_n$ are finite \'etale by construction. Then by the argument in the proof of 4.1.6 \cite{KL}, we are reduced to show (\ref{twonorms}) for $r=1, n=0$,  and may assume $f$ extends to  $f\in \cO(\bar{\fM}_{H^p}^{\rm ord})$ (by Lemma \ref{koecher}). In this case we may assume that   its image $f_0\in \cO(\bar{\fM}_{H^p}^{\rm ord})_0)$ is non-zero. If $|f|_1^{\circ}<|f|_1=1$, then $f_0$ is nilpotent on the open subset $(\fM_{H^p}^{\rm ord})_0$ of $(\bar{\fM}_{H^p}^{\rm ord})_0$. Thus it vanishes on $(\fM_{H^p}^{\rm ord})_0$ and hence on the whole $(\bar{\fM}_{H^p}^{\rm ord})_0$ by Theorem \ref{ordtor}, a contradiction. 
\end{proof}

 \begin{coro}\label{fk0}
Let $\theta: \pi_{2,\gamma_v}^*((\det\omega)^{k_0(p-1)})\ra
\pi_{1,\gamma_v}^*((\det\omega)^{k_0(p-1)})$ be the canonical map of
sheaves on $\fM_{H^pp^n}^{\gamma_v,\rm ur,rig}$. Then for
$r$ sufficiently close to $1$, 

\[\theta(\pi_{2,\gamma_v}^*(E))=f_{k_0}\pi_{1,\gamma_v}^*(E)
\in H^0(\fM_{H^pp^n}^{\gamma_v,\rm ur,rig}(r),
\pi_{1,\gamma_v}^*((\det\omega)^{k_0(p-1)}))\] for some $f_{k_0}\in
\cO(\fM_{H^pp^n}^{\gamma_v,\rm ur,rig}(r))$. Here we have used the same symbol $E$ to denote the pull-back of $E$ from  $\fM_{H^p}^{\rm rig}$ to its Galois cover $\fM_{H^pp^n}^{\rm ur,rig}$.

Moreover,  we have that $f_{k_0}-1$ is
topologically nilpotent. In fact, for any $0<\epsilon\leq 1$, there is an $r\leq 1$ such that $|f_{k_0}-1|_r\leq |p|^{\epsilon}$.
\end{coro}

\begin{proof}
 One sees that the proof of 4.1.7 of \cite{KL} applies. We only give a sketch here. As before, we reduce to the case $n=0$. The first assertion then follows from the fact that the special fibre $(\fM_{H^p}^{*,\rm ord})_0$ is normal, by Theorem \ref{normal}. To show the second assertion, first note that the  case $r=1$ follows from Lemma \ref{func} and Proposition \ref{norm}. The case for a general $r$ then follows from the $r=1$ case, as well as Proposition \ref{23}.
\end{proof}


Using Corollary \ref{fk0} we now show the following:

\begin{prop}\label{4110}

(1) We have the inclusion \[\pi_{2,\gamma_v}(\fM_{H^pp^n}^{\gamma_v, \rm ur, ord, rig})\subset \fM_{H^pp^n}^{\rm ur, ord, rig},\] and the inclusion 

\[\pi_{2,\gamma_v}(\fM_{H^pp^n}^{\gamma_v, \rm ur, rig}(r))\subset \fM_{H^pp^n}^{\rm ur, rig}(r)\] for $r$ sufficiently close to $1$. 

In particular, the inclusions above hold for $\pi_{2,\fa}$ on $\fM_{H^pp^n}^{\fa, \rm ur, ord, rig}$, with $\fa\subset \cO$ an ideal.
 
(2) For $r\ra 1^{-}$, $\pi_{2,(p)}$ induces a map \[\fM_{H^pp^n}^{(p),\rm ur,rig}(r^p)\ra \fM_{H^pp^n}^{\rm ur,rig}(r).\]

\end{prop}

\begin{proof}

(1) The first assertion follows from the construction of  $\fM_{H^pp^n}^{\rm ur, ord, rig}$. 

For the second inclusion,    we argue as  in the proof of 4.1.10 \cite{KL}. First observe that it is enough to show this for $E$-valued points for any finite extension $E/\QQ_p$. By the construction of $\fM_{H^pp^n}^{\gamma_v, \rm ur, rig}(r)$ we may assume $n=0$. Let $f: A\ra B$ be an element in $\fM_{H^p}^{\gamma_v, \rm ur, rig}(r)(E)$. We may enlarge $E$ so that $A$ extends to a semi-abelian scheme $\bar{A}$ over $\cO_E$.  We may and do assume that $\bar{A}$ is an abelian scheme, because otherwise $r=1$ and we go back to the first assertion. Then we can extend $\Ker(f)$ to a finite flat subgroup scheme of $\bar{A}$.  The quotient of $\bar{A}$ by this subgroup scheme is denoted by $\bar{B}$, and the projection $\bar{A}\ra \bar{B}$ by $pr$. Let $\mathbf{v}_{\bar{A}}$ (resp. $\mathbf{v}_{\bar{B}}$) be a basis of $H^0(\bar{A}, \Omega_{\bar{A}/\cO_E}^1)$ (resp. $H^0(\bar{B}, \Omega_{\bar{B}/\cO_E}^1)$). Then we must have \[pr^*(\det\mathbf{v}_{\bar{B}})=a\det\mathbf{v}_{\bar{A}}\] for some $a\in \cO_E$.

Now by the definition of overconvergence, it suffices to show \[|E(A, \det\mathbf{v}_{\bar{A}})|\leq |E(B, \det\mathbf{v}_{\bar{B}})|.\] By Corollary \ref{fk0}, we have \[E(B, \det\mathbf{v}_{\bar{B}})=f_{k_0}E(A, pr^*(\det\mathbf{v}_{\bar{B}})),\]with $f_{k_0}-1$ topologically nilpotent. On the other hand, we have the equality \[E(A, pr^*(\det\mathbf{v}_{\bar{B}}))=a^{(1-p)k_0}E(A, \det\mathbf{v}_{\bar{A}}).\] Now the result follows because $|f_{k_0}(A, \mathbf{v}_{\bar{A}})|=1$ and $|a^{(1-p)k_0}|\geq 1$.

(2) This is proved as in 4.3.3 \cite{KL}. By Proposition \ref{326}, we have for $s$ close enough to $1$ the natural map  \[\bar{\phi}_{n}(s): \fM_{H^pp^n}^{\rm ur, rig}(s)\ra \fM_{H^pp^n}^{\rm ur, rig}(s^p),\]by restricting to the intersection of the map in Proposition \ref{326} with $\cM_{H}^{\rm an}$. For $r\ra 1^{-}$, by part (1) we see $\pi_{2,(p)}$ induces a map \[\fM_{H^pp^n}^{(p),\rm ur,rig}(r^p)\ra \fM_{H^pp^n}^{\rm ur,rig}(s)\] for some $r^p\leq s$. If $s\geq r$, we are done. Now assume $s<r$. Thus it suffices to show that their composition 
\[\fM_{H^pp^n}^{(p),\rm ur,rig}(r^p)\stackrel{\pi_{2,(p)}}{\lra} \fM_{H^pp^n}^{\rm ur,ord,rig}(s)\stackrel{\bar{\phi}_{n}(s)}{\lra}\fM_{H^pp^n}^{\rm ur, rig}(s^p)\]
 factors through $\fM_{H^pp^n}^{\rm ur, rig}(r^p)\subset \fM_{H^pp^n}^{\rm ur, rig}(s^p)$, for which it is in turn enough to show the further composition with $\fM_{H^p}^{\rm rig}(s^p)$ 
 \[\fM_{H^pp^n}^{(p),\rm ur,rig}(r^p)\stackrel{\pi_{2,(p)}}{\lra} \fM_{H^pp^n}^{\rm ur,rig}(s)\stackrel{\bar{\phi}_{n}(s)}{\lra}\fM_{H^pp^n}^{\rm ur, rig}(s^p)\stackrel{\rm Ig_n}{\ra} \fM_{H^p}^{\rm rig}(s^p)\]
 factors through $\fM_{H^p}^{\rm rig}(r^p)$. 

Tracing the construction  we know that for $r\ra 1^{-}$ the latter composition is induced by the map  \[\fM_{H^pp^n}^{(p),\rm ur,ord}\stackrel{\pi_{1,(p)}}{\lra} \fM_{H^pp^n}^{\rm ur,ord}\stackrel{\rm Ig_n}{\lra}\fM_{H^p}^{\rm ord}\stackrel{[\cdot p]}{\lra}\fM_{H^p}^{\rm ord}, \]where the map $[\cdot p]$ is the multiplication by $p$ on the level structure of the universal abelian scheme. To see this, using Proposition \ref{23}, we only have to check the statement for the case $r=1$, which is easily seen. Meanwhile, we notice that   $[\cdot p]$  induces the identity map on Hasse invariant, because the Hasse invariant is independent of level structures.  Hence it maps  $\fM_{H^p}^{\rm rig}(r^p)$ to itself. This concludes the proof.  
\end{proof}

 \subsection{Hecke operators on overconvergent Siegel-Hilbert modular forms}
 
 Let $L/\QQ_p$ be a finite extension and $\cR$ an $L$-affinoid algebra with a fixed sub-multiplicative semi-norm extending the norm on $L$, and $Y\in \cR$ such that \[|Y|<|p|^{\frac{1}{p-1}-1}.\]  
 We use the convention that for a $\QQ_p$-analytic space $X$, $X_{\cR}=X\times_{\Sp \QQ_p}\Sp \cR$.  Define
  \[M_{H^pp^n,\kappa+Y,r}(L)=H^0(\fM_{H^pp^n}^{\rm ur,rig}(r)_{L},\omega^{\kappa}),\] 
   \[M_{H^pp^n,\kappa+Y,r}(\cR)=H^0(\fM_{H^pp^n}^{\rm ur,rig}(r)_{\cR},\omega^{\kappa}),\] 
 \[M_{H^pp^n, \kappa}^{\dagger}(L)=\varinjlim_{r\ra 1^{-1}}M_{H^pp^n,\kappa,r}(L),\]
  \[M_{H^pp^n, \kappa+Y}^{\dagger}(\cR)=\varinjlim_{r\ra 1^{-1}}M_{H^pp^n,\kappa+Y,r}(\cR),\]

 \begin{rk}
 Let $U_F^+$ be the group of totally positive units in $\cO_F$, and $U_{F,N}$ its subgroup of elements which are congruent to $1$ modulo $N$.  As noted in 1.11.8 \cite{KL} and Remark \ref{1118}, the moduli spaces we have constructed are for $G'=G\times_{\mathrm{Res}_{F/\QQ}\GG_m}\GG_m$, where the two maps in the product are the determinantal and diagonal ones. This is because multiplying by $U_F^+$ on the polarizations of tuples in $\cM_{H^pp^n}$   is an isomorphism for the subgroup $U_{F,N}^2\subset U_F^+$. We thus get an action of the finite group $U_F^{+}/U_{F,N}^2$ on $\cM_{H^pp^n}$, which induces a natural action on $H^0(\fM_{H^pp^n}^{\rm ur,rig}(r),\omega^{\kappa})$. Hence  the ${U_F^{+}/U_{F,N}^2}$-invariants of the spaces above are the spaces of forms on $G=\mathrm{Res}_{F/\QQ}\GSp_{2g}$.   
 \end{rk}

By abuse of notation, we call the spaces above the spaces of overconvergent Siegel-Hilbert modular forms of level $H^pp^n$ and of weight $\kappa$  (reps. $\kappa+Y$) with coefficients in  $L$ (reps.  $\cR$).

By Proposition \ref{4110} (1), we have two projections \[\pi_{1,\gamma_v}, \pi_{2,\gamma_v}:  \fM_{H^pp^n}^{\gamma_v,\rm
ur,rig}(r)\lra \fM_{H^pp^n}^{\rm
ur,rig}(r). \]We then get  the (pull-back of
) canonical map of sheaves on $\fM_{H^pp^n}^{\gamma_v,\rm
ur,rig}(r)_{\mathcal{R}}$
\[\pi_{2,\gamma_v}^*\omega^{\kappa}\lra
\pi_{1,\gamma_v}^*\omega^{\kappa},\] whose
composition with multiplication by $f_{k_0}^{\frac{Y}{k_0(p-1)}}$
gives the map
\[\pi_{2,\gamma_v}^*\omega^{\kappa}\lra
\pi_{1,\gamma_v}^*\omega^{\kappa}\stackrel{\cdot
f_{k_0}^{\frac{Y}{k_0(p-1)}}}{\lra}
\pi_{1,\gamma_v}^*\omega^{\kappa}.\]Here
\[f_{k_0}^{\frac{Y}{k_0(p-1)}}:
=\mathrm{exp}({\frac{Y}{k_0(p-1)}}\mathrm{log}f_{k_0})\] is a
well-defined element in
$\cO(\fM_{H^pp^n}^{\gamma_v, \mathrm{ur,rig}}(r)_{\mathcal{R}})$ for $r$
such that $|f_{k_0}-1|_r\leq |p|^{\epsilon}$, where $\epsilon$ is chosen to satisfy $|Y||p|^{\epsilon}<|p|^{1/(p-1)}$. We remark that the analyticity of $f_{k_0}^{\frac{Y}{k_0(p-1)}}$ follows from the assumption that $|Y|<|p|^{\frac{1}{p-1}-1}$ (and the assumption that $p\nmid k_0$).

Applying $\pi_{1,\gamma_v *}$ to the above map, and pre-composing it with the map from $\omega^{\kappa}
$ and composing it with the trace map, we obtain

\[\omega^{\kappa}\ra \pi_{1,\gamma_v *}\pi_{2,\gamma_v}^*\omega^{\kappa}\lra
\pi_{1,\gamma_v *}\pi_{1,\gamma_v}^*\omega^{\kappa}\stackrel{\pi_{1,\gamma_v *}\circ\cdot
f_{k_0}^{\frac{Y}{k_0(p-1)}}}{\lra}
\pi_{1,\gamma_v *}\pi_{1,\gamma_v}^*\omega^{\kappa}\stackrel{\rm tr}{\lra}
\omega^{\kappa}.\] Then taking global sections, we get an
endomorphism $T_{\gamma_v}$ of $H^0(\fM_{H^pp^n}^{\rm ur,rig}(r)_{\cR},\omega^{\kappa})$. The resulting Hecke operator will denoted by $T_{v,i}$ (resp. $S_v$) if $\gamma_v=T_{v,i}$ (resp. $S_v$).  

Letting $r\ra 1^-$,
we obtain the Hecke operator on
$M_{H^pp^n,\kappa+Y}^{\dag}(\mathcal{R})$, which is denoted by the same symbol. The ring of endomorphisms generated by all the Hecke operators (with $v$ and $\gamma_v$ varying) is denoted by $\mathbf{T}_{H^pp^n,\kappa+Y}^{\dag}$.  The product of $T_{v,1}$'s for all $v|p$ is denoted by $U_{(p)}$.

The  construction made above and Corollary \ref{fk0} thus give rise to the following

\begin{prop}\label{sep}

Let $\psi_t: \cR\ra L'$ be the character to a finite extension $L'/L$,  which sends $Y$ to $(p-1)k_0t$ for some $t\in \ZZ_{\geq 0}$. We have the following commutative diagram compatible with the actions of Hecke operators $T$:

\begin{equation}\label{classical}
\CD
  M_{H^pp^n,\kappa+Y}^{\dag}(\mathcal{R}) @>\mathrm{Id}\otimes\psi_t>> M_{H^pp^n,\kappa}^{\dag}(L') @>\cdot E^t>> M_{H^pp^n,\kappa\cdot\mathrm{Nm}^{(p-1)k_0t}}^{\dag}(L') \\
  @V T VV @. @V T VV  \\
 M_{H^pp^n,\kappa+Y}^{\dag}(\mathcal{R}) @>\mathrm{Id}\otimes\psi_t>> M_{H^pp^n,\kappa}^{\dag}(L') @>\cdot E^t>> M_{H^pp^n,\kappa\cdot\mathrm{Nm}^{(p-1)k_0t}}^{\dag}(L')  \\
\endCD\end{equation}
\end{prop}

\begin{proof}

Let $T=T_{\gamma_v}$. Then we are supposed to check that for $f\otimes x\in M_{H^pp^n,\kappa+Y}^{\dag}(\mathcal{R})$,
\numberwithin{equation}{subsection} \begin{equation}\label{speci}
\psi_t(T(f\otimes x))\cup E^t=T(f\otimes \psi_t(x)\cup E^t).\end{equation}(This suffices for the proof since the $T_{\gamma_v}$'s generate the Hecke algebra.)

Now unwinding the definition of the Hecke operator $T$, we see, for $A\ra B$ an isogeny in $\fM_{H^pp^n}^{\gamma_v,\rm ur, rig}(r)$, 
\[T(f\otimes x)(A)=f(B)\otimes f_{k_0}^{\frac{Y}{k_0(p-1)}}x.\]Then, applied to $(A\ra B)$, the left hand side of   the equality (\ref{speci}) is equal to  
\[f(B)\otimes \psi_t(x)\cup E^t(B) ,\]while the right hand side is equal to
\[(f(B)\otimes \psi_t(x))\cup f_{k_0}^tE^t(A), \]since $\psi_t(Y)=k_0(p-1)t$. Now the result follows from Corollary \ref{fk0}.
\end{proof}


\begin{prop}\label{res}

For $s<r<1$ with $s$ sufficiently close to $1$, the following natural inclusion   is completely continuous:

\[\mathrm{Res}(s,r): H^0(\fM_{H^pp^n}^{\rm ur,rig}(s),\omega^{\kappa})\hookrightarrow H^0(\fM_{H^pp^n}^{\rm ur,rig}(r),\omega^{\kappa}).\]

\end{prop}

\begin{proof}

It is equivalent to showing this for the natural inclusion \[H^0(\bar{\fM}_{H^pp^n}^{\rm ur,ord,rig}(s),\omega^{\kappa})\hookrightarrow H^0(\bar{\fM}_{H^pp^n}^{\rm ur,rig}(r),\omega^{\kappa})\] by the K\"{o}cher principle Lemma \ref{koecher}. 

Recall $\bar{\fM}_{H^pp^n}^{\rm ur,rig}(r)\Subset \bar{\fM}_{H^pp^n}^{\rm ur,rig}(s)$ from (\ref{forcontinuity}).  Now one concludes by Proposition 2.4.1 \cite{KL}.\end{proof}

\begin{lemma}\label{realize}
Suppose $r$ is close enough to
$1$. The Hecke operator $U_{(p)}$ on the $\mathcal{R}$-module
$M_{H^pp^n,\kappa+Y,r}(\mathcal{R})$ can be constructed as the
composition of  the natural inclusion $\mathrm{Res}(r^p,r)$ and the following  map induced by $\pi_{2,(p)}$:

\begin{equation}\label{p2}
H^0(\fM_{H^pp^n}^{\rm ur,rig}(r), \omega^{\kappa})\rightarrow H^0(\fM_{H^pp^n}^{\rm ur,rig}(r^p), \omega^{\kappa}).\end{equation}

\end{lemma}
\begin{proof}
This follows from  Proposition \ref{4110} (2).
\end{proof}

\begin{coro}\label{up}
For $r$ close enough to $1$,
the action of $U_{(p)}$ on $M_{H^pp^n,\kappa+Y,r}(\mathcal{R})$ is
completely continuous.
\end{coro}

\begin{proof}
We know by Proposition \ref{res} that the map $\mathrm{Res}(r,r^p)$ is completely continuous. Moreover, the map (\ref{p2}) is continuous. Since a composition of a continuous map followed by a completely continuous one is again completely continuous, we are done.
\end{proof}

\begin{rk}

It is the eigenvalues of the Hecke operator $U_{(p)}$ that we will interpolate, since we only construct a one-parameter family of overconvergent Siegel-Hilbert eigenforms. 

\end{rk}

 \subsection{Constructing families of overconvergent Siegel-Hilbert modular forms}

\subsubsection{The set up}

Recall that $T_{g}$ is the standard diagonal maximal torus of $\GSp_{2g/\mathbf{Z}}$. Denote by $c: T_g \rightarrow \mathbf{G}_m$ the character:
\begin{eqnarray*}
 c: \left( \begin{array}{rrrrrr}  a_1 & & & & & \\  & \ddots & & & & \\ & & a_g & & & \\ & & & b a_1^{-1} & & \\ & & & & \ddots & \\ & & & & & b a_g^{-1}    \end{array} \right) \mapsto a_1 \cdots a_g b^{-2}.
\end{eqnarray*}

Let $\mathcal{W}$ be the rigid space whose $E$-valued points are continuous homomorphisms in \[\Hom_{\rm cont}(T_g(\cO\otimes_{\ZZ}\ZZ_p),E^{\times})\] for any (not necessarily finite) field extension $E/\QQ_p$. 

In the rest of the paper, we fix a classical weight $\kappa$ and a finite extension $L/\QQ_p$. For our purpose we only need the part of the weight space that ``differs" from our fixed weight $\kappa$ by parallel weights. Thus let $\mathcal{W}_{\kappa}$ be the admissible subspace
of $\mathcal{W}$ whose $E$-valued points, for $E\subset \CC_p$ a
closed subfield containing $L$, are
\[\mathcal{W}_{\kappa}(E)=\{\chi=\kappa\cdot
(\tau\circ c \circ \mathrm{Nm}): T_g(\cO\otimes_{\ZZ}\ZZ_p)\ra E^{\times}\}\] for some continuous character
$\tau: \mathbf{Z}_p^{\times} \ra E^{\times}$ satisfying
\[v_p(1-\tau(t))>\frac{1}{p-1},\quad  t \in \mathbf{Z}_p^{\times}.\] 

We define a rigid analytic function $Y$ on $\mathcal{W}_{\kappa}$ as follows: if $\chi \in \mathcal{W}_{\kappa}(E)$ is as above, and is associated to $\tau:\mathbf{Z}_p^{\times} \rightarrow E^{\times}$, then the value of $Y$ at $\chi$ is given by
\[
Y(\chi) = \frac{\log \tau(t)}{\log t}
\]
for $t \in \mathbf{Z}_p^{\times}$ sufficiently close to the identity.

By the construction above, we have that  $|Y|<|p|^{\frac{1}{p-1}-1}$, hence the Banach module
$M_{H^pp^n,\kappa+Y}^{\dag}(\mathcal{R})$ of overconvergent forms is well-defined, for  any $\Sp \cR\subset \mathcal{W}_{\kappa}$. Let $\TT_{H^pp^n,\kappa+Y}^{\dag,\rm ur}$ be the closure of the ring of endomorphisms on $M_{H^pp^n,\kappa+Y}^{\dag}(\mathcal{R})$ generated by the Hecke operators at the places away from the level, under the norm defined in Proposition \ref{norm}. 
 
 \begin{prop}\label{comm}
 The $\mathcal{R}$-algebra $\TT_{H^pp^n,\kappa+Y}^{\dag,\rm ur}$ is commutative.
 \end{prop}

\begin{proof}	 


Let $\mathcal{W}_{\kappa}^{\rm cl}$ be a Zariski dense set of integral weights in the $Y$-neighbourhood of $\kappa$, which can be achieved by taking the parameters $t$ to be sufficiently large powers of $p$, by Proposition \ref{sep}. By the analyticity obtained in Proposition \ref{sep}, each element in  $\TT_{H^pp^n,\kappa+Y}^{\dag,\rm ur}$ is determined by the Zariski dense set $\mathcal{W}_{\kappa}^{\rm cl}$. We then have the injection 
\[\TT_{H^pp^n,\kappa+Y}^{\dag,\rm ur}\hookrightarrow \prod_{w\in \mathcal{W}_{\kappa}^{\rm cl}}\TT_{H^pp^n,w}^{\dag,\rm ur},\]where each factor is the specialization. On the other hand, each Hecke ring $\TT_{H,p^n,w}^{\dag,\rm ur}$ with the fixed integral weight $w$ is commutative, being the completion of a commutative algebra of Hecke correspondences. Thus the product over $\mathcal{W}_{\kappa}^{\rm cl}$ is commutative, so is its subring $\TT_{H^pp^n,\kappa+Y}^{\dag,\rm ur}$.
\end{proof}

Let \[Z_{\kappa}=   \Sp \TT_{H^pp^n,\kappa+Y}^{\dag,\rm ur}  \]be the
rigid space over $L$ associated to the $\mathcal{R}$-algebra $\TT_{H^pp^n,\kappa+Y}^{\dag,\rm ur}$. It comes with the weight map 
 \[\underline{w}: Z_{\kappa}\ra  \Sp \cR\subset \mathcal{W}_{\kappa}.\] 
Define
\[X_{\kappa}=Z_{\kappa}\times
\mathbb{G}_{m}.\]  Write $x_p$ for the
canonical co-ordinate on $\mathbb{G}_{m}$.

\subsubsection{Construction by the Coleman-Mazur machinery}

Now we are ready to define the one parameter families of overconvergent Siegel-Hilbert modular forms  of level
$H^pp^n$ as Coleman and Mazur proceed in \cite{CM}. For our purpose it is enough to 
construct it over any affinoid quasi-compact subset
$\mathrm{Sp}\mathcal{R}\subset \mathcal{W}_{\kappa}$. We fix such an $\cR$ from now on.

 Set $\mathcal{H}$ to be the (topological) commutative ring generated by the formal
variables $X_{(p)}$, together with $X_{v,i},Y_v$ (h ere $i=1,\cdots , g$) for all prime ideals $v \subset \cO$ away from the level. Let \[\iota:
\mathcal{H}\ra \cO(X_{\kappa}) \] be the map sending: 
\[X_{v,i} \mapsto t_{v,i},\]
\[Y_v \mapsto s_v,  \]
\[X_{(p)} \mapsto x_p. \]
Here we have denoted by $t_{v,i}$ and $s_v$ the image of the Hecke operators $T_{v,i}$ and $S_v$ in $\cO(X_{\kappa})$ respectively, regarded as functions on $Z_{\kappa}$. Then $\mathcal{H}$ acts on $M_{H^pp^n,\kappa+Y}^{\dag}(\mathcal{R})$ with the action factoring through $\iota(\mathcal{H})$.

For $r$ sufficiently close to $1$, we know by Corollary \ref{up} that $U_{(p)}$
acts on $(M_{H^pp^n,\kappa+Y,r})(\mathcal{R})$ completely
continuously. This implies that the action of $\iota(\alpha)U_{(p)}$ on
$M_{H^pp^n,\kappa+Y,r}(\mathcal{R})$ is completely continuous
for any $\alpha\in \mathcal{H}$. Following Section 4 of \cite{CM}, we
can, for each $\alpha\in \mathcal{H}$, form the Fredholm series
\[P_{\alpha}(T)=\mathrm{det}_{\mathcal{R}}(1-\iota(\alpha)U_{(p)}T|M_{H^pp^n,\kappa+Y,r}(\mathcal{R}))\in \cR[[T]],\]which
is independent of the choice of $r$ (for $r$ sufficiently close to
$1$) by the lemma  below.

\begin{lemma}\label{independ}
Let $0<r<r'<1$ with $r$ sufficiently close to $1$. Then the
Banach $\mathcal{R}$-module $M_{H^pp^n,\kappa+Y,r}(\mathcal{R})$
admits an orthogonal basis, which is also an orthogonal basis for
the $\mathcal{R}$-submodule $M_{H^pp^n,\kappa+Y,r'}(\mathcal{R})$.
\end{lemma}

\begin{proof}
Since
$M_{H^pp^n,\kappa+Y, r}(\mathcal{R})=M_{H^pp^n,\kappa+Y, r}(\QQ_p)\hat{\otimes}_{\QQ_p}\mathcal{R}$,
we may assume $\mathcal{R}=\QQ_p$ and $Y=0$. Recall that, as $r$ is
sufficiently close to $1$, the natural map $\fM_{H^pp^n}^{\rm
*,ur,rig}(r) \ra \fM_{H^p}^{*,\rm rig}(r)$ is finite \'{e}tale. We can
conclude the proof by 2.4.5 \cite{KL}, namely in the notation of {\it loc. cit.} we let $\mathcal{F}$ be the push-forward of $\omega^{\kappa}$
under the composition (recall the notation from Section \ref{igusa}) \[\bar{\fM}_{H^pp^n}^{\rm ur,rig}(r) \ra \fM_{H^pp^n}^{\rm
*,ur,rig}(r) \ra \fM_{H^p}^{*,\rm rig}(r)\]which is proper. Moreover, we have the line bundle
$\mathcal{L}=(\det\omega)^{p-1}$ which is ample over $\cM_{H^p}^*$, and $\mathcal{D}\subset
(\cM_{H^p}^{*})_{\mathbf{F}_p}$ the divisor where the Hasse invariant $h$ vanishes. Then 2.4.5 \cite{KL} gives the result we require. \end{proof}

Set \[\mathcal{E}_{\kappa}=\mathcal{E}_{\kappa}^{\rm red}\subset X_{\kappa}\] to be the nilreduction of the 
Zariski-closed subspace of $X_{\kappa}$ cut out by the ideal
generated by the functions $P_{\alpha}((x_p\iota(\alpha))^{-1})$ for
all the $\alpha\in \mathcal{H}$ such that $\iota(\alpha)$ is a unit. Alternatively, we can define $\mathcal{E}_{\kappa}$ as follows. 

The entire series associated to $\alpha\in \mathcal{H}$, $P_{\alpha}(T)\in \cR[[T]]$, defines a closed  subspace \[\mathcal{Z}_{\alpha}\subset \Sp\cR\times \mathbf{A}^1,\] where $T$ is regarded as the co-ordinate on $\mathbf{A}^1$. For each $\alpha\in \mathcal{H}$ such that $\iota(\alpha)$ is unit, we can define the map 

\[r_{\alpha}: X_{\kappa}=Z_{\kappa}\times \mathbf{G}_m\ra \Sp\cR\times \mathbf{A}^1: x=(z, s)\mapsto (\underline{w}(z), \frac{s}{(\iota(\alpha))(x)}),\]where we have regarded $\iota(\alpha)$ as a function on $X_{\kappa}$.  Then we set $\mathcal{E}_{\kappa}$ to be the nilreduction of 

\[\bigcap_{ \substack{\alpha\in \mathcal{H}, \\  \iota(\alpha)\in \cO(X_{\kappa})^{\times} } } r_{\alpha}^{-1}(\mathcal{Z}_{\alpha}).\]

The following theorem is obtained from the above construction
formally, as in \cite{CM}.

\begin{theo}\label{curve}

(1) Let $E\subset \CC_p$ be a closed subfield containing $L$. For an
$E$-valued point $x\in \mathcal{E}_{\kappa}(E)$, there is a non-zero
simultaneous eigenvector $f_x\in
M_{H^pp^n,\kappa+Y(x)}^{\dag}(E)$ for all the Hecke operators in
$\TT_{H^pp^n,\kappa+Y}^{\dag}$ such that the Hecke eigenvalues $\lambda_{T_{v,i}}(z),\lambda_{S_v}(x),\lambda_{U_{(p)}}(x)$ for the operators $T_{v,i},S_v$ and $U_{(p)}$ satisfy: 
\[\lambda_{T_{v,i}}(x)= t_{v,i}(x),\]
\[\lambda_{S_v}(x)=s_v(x), \]
\[\lambda_{U_{(p)}}(x)=x_{p}(x).\]For a fixed $Y_0\in E$ with $v_p(Y_0-1)>\frac{1}{p-1}$, the above
assignment induces a bijection between the points $\{x\in
\mathcal{E}_{\kappa}(E)\}_{Y(x)=Y_0}$ and systems of
$\TT_{H^pp^n,\kappa+Y}^{\dag}$-eigenvalues of an eigenvector
$f\in M_{H^pp^n,\kappa+Y_0}^{\dag}(E)$ of \emph{finite slope} at $p$.

(2) The rigid analytic space $\mathcal{E}_{\kappa}$ is a curve.
The weight map $\underline{w}: \mathcal{E}_{\kappa}\ra \mathcal{W}_{\kappa}$ is,
locally in the domain, finite flat. The image of any component of
$\mathcal{E}_{\kappa}$ under this map misses at most finitely many points in
$\mathcal{W}_{\kappa}$.

\end{theo}

The following theorem plays a similar (yet weaker) role as the expected result that classical Siegel-Hilbert  eigenforms are Zariski dense in the rigid
analytic space $\mathcal{E}_{\kappa}$.

\begin{theo}\label{accum}
Let $f$ be a classical Siegel-Hilbert modular eigenform of weight $\kappa$ and
of level $H^pp^n$. There exists, for any positive integer $t$ with $v_p(t)$ large enough, a Siegel-Hilbert modular
eigenform $f_t$ of weight $\kappa\cdot \mathrm{Nm}^{(p-1)k_0t}$ and of the same level, such
that the Hecke eigenvalues on the $f_t$'s converge $p$-adically to that of
$f$, as $v_p(t)\ra +\infty$. Furthermore, if $f$ is cuspidal, then the $f_t$ can also be taken to be cuspidal.
\end{theo}
\begin{proof}The proof is completely similar to that of 4.5.6  \cite{KL}.

As before, we take the weight space $\mathcal{W}_{\kappa}$ centered in $\kappa$ to be $\Sp \cR$ since the construction is local. By the construction of $\mathcal{E}_{\kappa}$, when $\iota(\alpha)$ is  unit, we have a map 
\[r_{\alpha}: \mathcal{E}_{\kappa}\lra \mathcal{Z}_{\alpha}.\]
By the method of Chapter 7 \cite{CM}, we see the projection  $r_{\alpha}$ is finite. 

Let $x\in \mathcal{E}_{\kappa}(L)$ be the point corresponding to $f$. By the argument of 6.2.2 and 6.3.2 \cite{CM} we may assume 
\numberwithin{equation}{subsection} \begin{equation}\label{locfin}
r_{\alpha}^{-1}(r_{\alpha}(x))=\{x\}.\end{equation}By this property, we only need to find a family of elements in $\mathcal{Z}_{\alpha}(L)$ converging to $r_{\alpha}(x):=x_0$.

Let $w\in \Sp \cR$ denote the weight of $x$. Let $x_1,\cdots, x_r$ be the points in $\cZ_{\alpha}$ which lie over the weight $w$ and correspond to other (finitely many by 1.3.7 \cite{CM}) roots of $P_{\alpha}(T)_w\in L[T]$, the specialization  by $w$ of $P_{\alpha}(T)\in \cR[[T]]$.  The $\iota(\alpha)U_{(p)}$-eigenvalue of $x_i$ ($0\leq i\leq r$) is denoted by $\lambda_i$.   By the (local) finite flatness of the weight map $\underline{w}$ (shrinking $\Sp\cR$ if necessary) we may assume there are disjoint connected components $\{\cZ_i\}_{i=0,\cdots,r}$ of $\mathcal{Z}_{\alpha}$, such that for any $0\leq i\leq r$,  

\begin{itemize}
\item $x_i$ is the only point in $\cZ_i$ among the points $\{x_0,\cdots,x_r\}$.
\item  $T/\lambda_0$ is topologically unipotent on $\cZ_0$, and is topologically nilpotent on  $\cZ_i$ for any $i\geq 1$. 
\item  $\cZ_i$ is finite over $\Sp\cR$. 
 \end{itemize}

Thus $\bigcup_{i=1}^{r} \cZ_i$ is finite flat over $\Sp \cR$, hence corresponds to a polynomial $F(T)\in \cR[[T]]$ dividing $P_{\alpha}(T)$. By the construction, we have the following well-defined idempotent operator 
\[e=\lim_{n\ra \infty}\left(\frac{\iota(\alpha) U_{(p)}}{\lambda_0}\right)^{n!}\frac{F(\alpha U_{(p)})}{F(\lambda_0)^{-1}}, \] which is easily checked to be the  identity on a point in $\cZ_0$ and kill any points in  $\cZ_i$ for $i\geq 1$. 

Consider the integers $t$ such that $Y^{-1}(Y(x)+(p-1)k_0t)\in \Sp \cR$. We form the  Siegel-Hilbert modular eigenform of level $H^pp^n$ and weight $\kappa+(p-1)k_0t$:

\[g_t=e(E^t\cdot f)=e(E^t\cdot g_0).\]

By Proposition \ref{sep} (applying the first row of diagram (\ref{classical}) to $f$) and the continuity of the Hecke action on $M_{H^pp^n,\kappa+Y}^{\dag}(\cR)$, we have that 
\[g_t\neq 0,  \text{ if } v_p(t)\gg 0.\]We can write $E^t\cdot g_0$ as a finite sum of classical eigenforms. If $f$ is cuspidal to begin with, then $E^t\cdot g_0$ can be written as a finite sum of cuspidal eigenforms. Pick one of them so that the associated point $x_t\in \mathcal{E}_{\kappa}$ has image in $\cZ_0$ under the projection $r_{\alpha}$. By this construction, the point $x$ is the limit of $r_{\alpha}(x_t)$  when $t$ goes to $0$ $p$-adically. Now by (\ref{locfin}), one has that $x$ is the limit of $x_t$. Let $f_t$ be the classical Siegel-Hilbert modular form corresponding to $x_t$. This finally concludes the proof.  
\end{proof}

\subsection{Complement}\label{des}

In this final subsection we give a complement to theorem 4.15, which is needed for the application \cite{Mok}. 

Thus let $v$ be a prime of $\mathcal{O}$ with $v \nmid p$. Fix a Bernstein component $\mathcal{B}_v$ of $\GSp_{2g}(F_v)$. Recall that $\mathcal{B}_v$ is given by the (equivalence class of) data given by a pair $(M,\tau)$, where $M$ is a Levi subgroup of $\GSp_{2g/F_v}$, and $\tau$ is a supercuspidal representation of $M(F_v)$, up to twisting by unramified characters of $M(F_v)$. Let $E$ be a number field over which $\mathcal{B}_v$ is defined, and denote by $\mathfrak{z}_v = E[\mathcal{B}_v]$ the affine coordinate ring of $\mathcal{B}_v$, which is known as the Bernstein centre of $\mathcal{B}_v$. We have an idempotent element $e_v \in \mathfrak{z}_v$, such that for any irreducible admissible representation $\pi_v$ of $\GSp_{2g}(F_v)$, we have $\pi_v$ belongs to the component $\mathcal{B}_v$ if and only if $e_v \cdot \pi_v \neq 0$. 

Now we come back to the context of the previous subsections. Let $l$ be the rational prime below the prime $v$, and let $m$ be the exact power such that $l^m$ divides $N$. We assume that $n$ is sufficiently large, so that the following holds: denoting by $K_v(m)$ the principal congruence subgroup at the prime $v$ of level $n$, with associated idempotent $e_{K_v(m)}$. Then we assume that $n$ is large enough so that
\[
e_{K_v(m)} \cdot e_v =e_v
\]  	 
with $e_v$ being the idempotent associated to the Bernstein component $\mathcal{B}_v$ as above. We may also assume that the extension $L/\mathbf{Q}_p$ in the last subsection to be large enough to contain the number field $E$.

We note that in the particular case where $\mathcal{B}_v$ is the Iwahori component associated to the (standard) Iwahori subgroup $I_v \subset \GSp_{2g}(F_v)$, then for any $\pi_v$ belonging to $\mathcal{B}_v$, we have $U_{(p)}$ acts invertibly on $\pi_v^{I_v}$ ({\it c.f.} section 6.4.1 of \cite{BC}).

Back to the fixed Bernstein component $\mathcal{B}_v$ as above. The Bernstein centre $\mathfrak{z}_v$ acts on the space of overconvergent Siegel-Hilbert modular forms $M^{\dag}_{H^pp^n,\kappa+Y}(\mathcal{R})$. Indeed by the theory of Bernstein centre it suffices to see that the local Hecke algebra of $\GSp_{2g}(F_v)$ with respect to the congruence subgroup $K_v(m)$ acts on $M^{\dag}_{H^pp^n,\kappa+Y}(\mathcal{R})$. Since $v \nmid p$, this follows by a similar argument as in section 4.2.  
	 
As in chapter 7 of \cite{BC}, we can then form the space of overconvergent Siegel-Hilbert modular forms associated to the idempotent $e_v$:
\[
e_v M^{\dag}_{H^pp^n,\kappa+Y}(\mathcal{R}).
\]

Then the same argument as in the proof of theorem 4.15, but with the constructions applied to the space $e_v M^{\dag}_{H^pp^n,\kappa+Y}(\mathcal{R})$, yields the following: 	 
	 
\begin{theo}\label{accum}
Let $f$ be a classical cuspidal Siegel-Hilbert modular eigenform of weight $\kappa$ and
of level $H^pp^n$. Let $\pi$ be the cuspidal automorphic representation of $\GSp_{2g}(\mathbf{A}_F)$ generated by $f$, and assume that $\pi_v$ belongs to $\mathcal{B}_v$. There exists, for any positive integer $t$ with $v_p(t)$ large enough, a Siegel-Hilbert cuspidal
eigenform $f_t$ of weight $\kappa\cdot \mathrm{Nm}^{(p-1)k_0t}$ and of the same level, such
that the Hecke eigenvalues on the $f_t$'s converge $p$-adically to that of
$f$, as $v_p(t)\ra +\infty$. Furthermore the $f_t$ can be taken to have the following property: denoting by $\pi_t$ the cuspidal automorphic representation of $\GSp_{2g}(\mathbf{A}_F)$ generated by $f_t$. Then $\pi_{t,v}$ belongs to $\mathcal{B}_v$.
\end{theo}

\bigskip	 
	  
\textit{Acknowledgement}
\medskip

We are grateful to Kiran Kedlaya,  Kai-Wen Lan and George Pappas for helpful discussions. The second-named author was supported through the program ``Oberwolfach Leibniz Fellows'' by the Mathematisches Forschungsinstitut Oberwolfach in 2011, where an early draft of the paper was written.


\vspace{\baselineskip}

Chung Pang Mok

Department of Mathematics and Statistics

McMaster University 

Hamilton, Ontario L8S 4K1

Canada

cpmok@math.mcmaster.ca

\vspace{\baselineskip}
 \vspace{\baselineskip}

Fucheng Tan

Department of Mathematics

Michigan State University

East Lansing, Michigan 48824

USA

ftan@math.msu.edu

\end{document}